\documentclass[makeidx,reqno]{amsart}
\usepackage[active]{srcltx}
\usepackage{epsfig}
\usepackage{hyperref}
\usepackage{amsmath}
\usepackage{amssymb}
\newtheorem{Proposition}{Proposition}
  \newtheorem{Remark}{Remark}
  
  \newtheorem{Lemma}[Proposition]{Lemma}
  
  \newtheorem{Theorem}{Theorem}
 
 \newtheorem{Definition}[Proposition]{Definition}
 \newtheorem{Note}[Remark]{Note}
 
\newtheorem{Assumptions}{Assumption}
\newtheorem{Normalization}{Normalization}
\newcommand {\z}{{\noindent}}
 
\def\CC{\mathbb{C}}
 \def\RR{\mathbb{R}}
 \def\NN{\mathbb{N}}
\def\ZZ{\mathbb{Z}}
\def\QQ{\mathbb{Q}}
\def\Re{\mathrm{Re}}
\def\Im{\mathrm{Im}}
\def\ds{\displaystyle}

\def\bb{b} \def\({\left(} \def\){\right)} \makeindex
\author{O. Costin, M. Huang} \address{Mathematics Department//The Ohio State University//Columbus, OH 43220} \title{Behavior of lacunary
  series at the natural boundary} 
\begin{document}
\begin{abstract}
  We develop a local theory of lacunary Dirichlet series of the form
  $\sum\limits_{k=1}^{\infty}c_k\exp(-zg(k)), \Re(z)>0$ as $z$
  approaches the boundary $i\RR$, under the assumption $g'\to\infty$
  and further assumptions on $c_k$. These series occur in many
  applications in Fourier analysis, infinite order differential
  operators, number theory and holomorphic dynamics among others. For
  relatively general series with $c_k=1$, the case we primarily focus
  on, we obtain blow up rates in measure along the imaginary line and
  asymptotic information at $z=0$.

  When sufficient analyticity information on $g$
  exists, we obtain Borel summable expansions  at points on
  the boundary, giving exact local
  description.  Borel
  summability of the expansions provides property-preserving
  extensions beyond the barrier.

The singular
behavior has remarkable universality and self-similarity features. If
$g(k)=k^b$, $c_k=1$, $b=n$ or $b=(n+1)/n$, $n\in\NN$, 
behavior near the boundary is roughly of the standard form $\Re(z)^{-b'}Q(x)$ where
$Q(x)=1/q$ if $x=p/q\in\QQ$ and zero otherwise.

The B\"otcher map at infinity of polynomial iterations of the form
$x_{n+1}=\lambda P(x_n)$, $|\lambda|<\lambda_0(P)$, turns out to have
uniformly convergent Fourier expansions in terms of simple lacunary
series. For the quadratic
map $P(x) =x-x^2$, $\lambda_0=1$, and the Julia set is the graph of this Fourier expansion in the main cardioid of the Mandelbrot set.

\end{abstract}
\maketitle
\tableofcontents
\section{Introduction}

Natural boundaries (NBs) occur frequently in many
applications of analysis, in the theory of Fourier series, in
holomorphic dynamics (see \cite{Douady}, \cite{Devaney}, \cite{CPAM}
and references there), in analytic number theory, see
\cite{Titchmarsh}, physics, see \cite{Kawai} and even in relatively simple ODEs such as the
Chazy equation \cite{Ince}, an equation arising in conformal mapping
theory, or the Jacobi equation.

The intimate structure of NBs turns out to be particularly rich,
bridging analysis, number theory and complex dynamics.

Nonetheless (cf. also \cite{Kawai}), the study of NBs of concrete functions is yet to
be completed from a pure analytic point of view. The aim of the present paper is a detailed study near the analyticity
boundary of a prototypical  functions exhibiting this 
singularity structure, classes of {\em lacunary series}.  For such functions, we
develop a theory of generalized local asymptotic expansions at NBs,  and explore their consequences and applications. The expansions are asymptotic in the sense
that they become increasingly accurate as the singular curve is approach, and in many cases exact, in that the function can be recovered from these expansions.

Lacunary series,  sums of the form $h(s)=\sum_{j\ge 1}{c_j}s^{g_j}$, or written
as Dirichlet series $f(z)=\sum_{j\ge 1}{c_j}e^{-z{g(j)}}$), where
$j/g_j=o(1)$ for large $j$, often occur in applications and have deep
connections with infinite order differential operators, as found and studied  by
Kawai \cite{Kawai1}, \cite{Kawai}.  Under the lacunarity assumption
above, if the unit disk is the maximal disk of analyticity of $h$,
then the unit circle is its NB (\cite{Mandelbrojt}). For
instance, the series
\begin{equation}
  \label{eq03}
 h(s)=\sum_{j=1}^{\infty}s^{2^j}\:(|s|<1)
\end{equation}
studied by Jacobi \cite{Jacobi}
 before the advent of modern Complex Analysis, clearly has the
 unit disk as a singular curve:  $h(s)\to+\infty$ as $|s|\uparrow
 1$ along any ray of angle $2^{-n}m\pi $ with $m,n\in\NN$.

We show that if $c_j=1$ and $g'\to\infty$,
then, in a measure theoretic sense,  $f(x+iy)$ blows up
as $x\to 0^+$ uniformly in $y$ at a calculable rate. We find
interesting {\em universality properties in the blow-up profile}.

In special cases of interest, Borel summable power series, in powers
of the distance to the boundary,  and more generally convergent 
expansions
as series of small exponentials multiplying  Borel summed series power series\footnote{These are simple instances of Borel (or Ecalle-Borel) summed transseries. A brief summary of the definitions and properties
of transseries and generalized Borel summability, and references to
the literature are given in \S\ref{A4}. \label{f1}}
representations
can be determined on a dense set on the singularity
barrier. Examples are
\begin{itemize}

\item $\sum_{j\ge 1}e^{-z j^b}$ where $b>1$, or its dual $q>1$, where
  $ b^{-1}+q^{-1}=1$, is integer (relating to exponential sums and van
  der Corput dualities \cite{Montgomery}; the special self-dual case
  $\sum_{j\ge 1}e^{-z j^2}$, is related to the Jacobi theta function);

\item $\sum_{j\ge 1}e^{-z a^j}$, $1<a\in\NN$;
\end{itemize}
More generally, if $c_j=c(j)$ and $g_j=g(j)$ have suitable analyticity
properties in $j$, then the behavior at $z=0^+$, and possibly at other
points, is described in terms of Ecalle-Borel summed expansions (see
footnote \ref{f1}).  Then the analysis leads to a natural,
properties-preserving, continuation formulas across the boundary.

In general, the blow-up profile along the barrier is closely related to
{\em exponential sums,} expressions of the form
\begin{equation}
  \label{eq:es1}
S_N=\sum_{k=1}^N c_n\exp(2\pi i g(n)),\ \ g(n)\in\RR
\end{equation}
where for us $g'\to\infty$ as $n\to\infty$. The corresponding lacunary
series are in a sense the continuation of (\ref{eq:es1}) in the
complex domain, replacing $2\pi i $ by $-z$, $\Re(z)>0$, and letting
$N\to \infty$. The asymptotic behavior of lacunary series as the
imaginary line is approached in nearly-tangential directions is
described by dual, van der Corput-like, expansions. The method we use
extend to exponential sums, which will be the subject of a different
paper.

\section{Results}
\subsection{Results under general assumptions}
\subsubsection{Blow-up on a full measure set}
We consider lacunary Dirichlet series  of the form
\begin{equation}
  \label{lacn1}
 f(z)= \sum_{k=0}^\infty e^{-zg(k)}
\end{equation}
but, as it can be seen from the proofs, the analysis extends easily to series of the form
$$f(z)= \sum_{j=0}^\infty  c_j e^{-zg(j)}$$
under suitable smoothness and growth conditions, see
\S\ref{genblowup}. The results in the paper apply under the further restriction,
\begin{Assumptions}\label{A1}
  The function $g$ is  differentiable and  $g'(j)\to \infty$ as $j\to\infty$.
\end{Assumptions}

In particular, $g$ is eventually increasing. By subtracting a finite
sum of terms from $f$ (a finite sum is clearly entire), we arrange that $g$ is
{\bf increasing}.  If $g(0)=a$, we can multiply $f$ by $e^{-za}$ to
arrange that $g(0)=0$.

\begin{Normalization}

(i) $g$ is differentiable on $[0,\infty)$, $g'>0$ and $g'\to\infty$
along $\RR^+$.

(ii) $g(0)=0$.
\end{Normalization}

\bigskip

\z {\bf Notation.} We write  $\ds |H(\cdot)|\mathop{=}^\mu 1+o(1)$ if $|H(y)|dy$ converges to
the Lebesgue measure $d\mu(y)$.

\bigskip

Under Assumption \ref{A1}, after normalization, we have the following result,
giving exact blow-up rates in measure, as well as sharp pointwise  blow-up upper bounds.
 \begin{Theorem}\label{P1}
(i)  We have the uniform
blow-up rate in measure\footnote{It turns out that in general, $L^1$ or  a.e. convergence of $|f|$ do
not hold.}.
\begin{equation}
  \label{eq:asbh}
  |f(x+i\cdot)|^2\mathop=^\mu\int_{0}^{\infty}e^{-2g(s)x}ds\left(1+o(1)\right)
\end{equation}
It can be checked that  $ \int_{0}^{\infty}e^{-2g(s)x}ds\ \ge g^{-1}(1/x)\to\infty {\text{ as }} x\to 0^+$;  see also Note~\ref{n1}.

(ii) The following pointwise estimate holds:
\begin{equation}
  \label{eq:locb}
  \|f(x+i\cdot)\|_{\infty}\le \int_{0}^{\infty}e^{-g(s)x}ds(1+o(1))\ \ \text{as } x\to 0^+
\end{equation}
This is sharp at $z=0$, cf. Proposition \ref{Prop2},  and
in many cases it is only reached at $z=0$; see Proposition~\ref{Pp4}.
\end{Theorem}

\subsubsection{General behavior near  $z=0$}
At $z=0^+$ a more detailed asymptotic description is possible.
\begin{Theorem}\label{Prop2}
  (i) As $z\to 0^+$  we have
  \begin{equation}
    \label{eq:asympt1}
    f(z)=\int_0^{\infty}e^{-z g(s)}ds-\frac{1}{2}+o(1)\ \to\infty\ \ \ \text{as } z\to 0^+
  \end{equation}
In fact,
\begin{equation}
    \label{eq:asympt1}
   f(z)-\int_0^{\infty}e^{-z g(s)}ds=-z\int_{0}^{\infty} e^{-zu}
   \{g^{-1}(u)\}du
  \end{equation}
(where $\{\cdot\}$ denotes the fractional part, and we used $g(0)=0$)\footnote{As mentioned, often this maximal growth
    is achieved at zero but in special cases it occurs, up to a
    bounded function, on a dense set of measure zero.}.

(ii)  If $g(s)$ has a
  differentiable asymptotic expansion  as $s\to\infty$ in terms of (integer or noninteger)
powers of $s$ and $\log s$, and $g(s)\sim const. s^b $, $b>1$, then after
subtracting the blowing up term, $f$
has a Taylor series at $z=0$ (generally divergent, even when $g$ is analytic,
which can be calculated explicitly),
$$
f(z)=
\int_0^{\infty}e^{-z g(s)}ds-\frac{1}{2}+zs(z), \ \ s\in C^\infty[0,\infty)$$
(as an example, see (\ref{eq:srjtob}).
\end{Theorem}

\begin{Note}\label{n1}
 Often  $g$ has
an asymptotic expansion starting with a combination of powers, exponentials and
logs.   Let $\phi=g^{-1}$. Then $\phi(\nu x)/\phi(\nu)\to \phi_1(x)$ as $\nu\to\infty$ and $x>0$ is
fixed and
$$\int_0^{\infty}e^{-x g(s)}ds=C_gg^{-1}(1/x)(1+o(1));\ \ \ C_g=\int_0^{\infty}e^{-u}\phi_1(u)du$$
\end{Note}
 For instance, if  $g(k)=k^b,b>1$ we have, as $\rho\to 1$,
\begin{eqnarray}
  \label{nb}
 |f(z)|\mathop=^\mu x^{-\frac{1}{2b}} \,\,\Gamma(1+1/b)^{-1/2}2^{-\frac{1}{2b}}(1+o(1)) \ \ \ (x\to 0^+)\nonumber \\
 f(x)=x^{-\frac{1}{b}} \,\, \Gamma(1+1/b)(1+o(1))\ \ \ \ \ \ \ (x\to 0^+)
\end{eqnarray}

\subsubsection{Blow-up profile along barrier}

\begin{Theorem}\label{Cexp}
  (i) Assume that for some $y\in\RR$ there is a smooth increasing
  function $\rho(N;y)=:\rho(N)\in(0,N]$ such that the following
  weighted exponential sum (see \cite{Montgomery}) has a limit:
  \begin{equation}
    \label{eq:expos}
   S_{\rho,N}:= \rho(N)^{-1}\sum_{j=1}^N e^{i y g(j)}\to L(y)\ \ \text{as}\ \ N\to \infty
  \end{equation}
   where $\rho''$ is uniformly bounded and nonpositive for sufficiently
  large $k$. (Without loss of generality, we may assume $\rho''(k)\le
  0$ for all $k$.)  Let
  \begin{equation}
    \label{eq:H1}
    \Phi(x)=\int_0^{\infty}e^{-xg(u)}\rho'(u)du
  \end{equation}
Then, we have the asymptotic behavior
  \begin{equation}
    \label{eq:lim2}
    f(x+iy)=L(y)\Phi(x)+o(\Phi(x))\ \ \ (x\to 0^+)
  \end{equation}
(ii) As a pointwise upper bound we have:
\begin{equation}
  \label{eq:lsup}
  \limsup_{N\ge 0} |S_{\rho,N}|=L<\infty\ \Rightarrow  \ f(x+iy)=O(\Phi(x))\ \ \ (x\to 0^+)
\end{equation}
\end{Theorem}\label{P4}

\subsection{Results in specific cases}

\subsubsection{Settings leading to convergent expansions}
The cases $g(j)=j^2$ and  $g(j)=a^j$, $a>1$ are distinguished, since
the expansions at some points near the boundary converge.

\begin{Proposition}\label{b=2}
\z If $b=2$, then we have the identity \[
f(z)=\frac{1}{2}\sqrt{\frac{\pi}{z}}-\dfrac{1}{2}+\sqrt{\frac{\pi}{z}}\sum_{k=1}^{\infty}e^{-\frac{k^{2}\pi^{2}}{z}}\]

\end{Proposition}
Clearly, this is most useful when $z\to 0$. It also shows the identity
associated to the Jacobi theta function
\begin{equation}
  \label{eq:ide2}
  f(z)=\frac{1}{2}\sqrt{\frac{\pi}{z}}-\dfrac{1}{2}
+\sqrt{\frac{\pi}{z}}f\left(\frac{k^2\pi^2}{z}\right)
\end{equation}
\begin{Proposition}\label{Prop4}
  If $g(j)=a^j$, $a>1$,  then, as $z\to 0^+$, $f(z)$
  is convergently given by
 \begin{multline}\label{convt}
  f(z)=-\dfrac{\log \zeta}{\log a}+\sum_{n=1}^{\infty}\frac{(-\zeta )^{n}}{n!(1-a^{n})}+c_{0}+\frac{1}{\log a}\sum_{k\neq0}\Gamma\left(-\frac{2k\pi i}{\log a}\right)\zeta ^{\frac{2k\pi i}{\log a}}\\=-\dfrac{\log \zeta }{\log
a}+\sum_{n=1}^{\infty}\frac{(-\zeta )^{n}}{n!(1-a^{n})}
+c_{0}-
\frac{1}{2\pi i}\int_{0}^{\infty}\log_{\RR}[-(s/\zeta )^{\frac{2\pi i}{\log a}}]e^{-s}ds
 \end{multline}
where $\zeta=1-e^{-z}$.
(Here $\log_{\RR}$ is the usual branch of the log with a cut along
 $\RR^-$ and {\bf not} the log on the universal covering of
 $\CC\setminus \{0\}$.)

It is clear that for $a\in\NN$ the transseries
can be easily calculated for any $z=\rho\exp(2\pi i m/a^j)$, $(m,j)\in\NN, 0\le \rho<1$ since
$$f(\rho \exp(2\pi i m/a^j))=\sum_{n=1}^{j} \rho^{a^j}\exp(2\pi i m a^{n-j})+f(\rho)$$
where the sum is a polynomial, thus analytic.
\end{Proposition}

\subsubsection{Borel summable transseries representations; resurgence (cf. \S\ref{A4})}

When $c_j\equiv 1$ it is clear that the growth rate as $z\to 0^+$
majorizes the rate at any point on $i\RR$. There may be {\bf no} other
point with this growth, as is the case when $g(j)=j^b$, $b\in (1,2)$
as seen in
see Proposition \ref{Pp4} below, or densely many if, for instance,
$g(j)=j^b$, $b\in\NN$, or when $g(j)=a^j$, $a\in\NN$. The behavior
near points of maximal growth merits special attention.

For $g=j^b$, $1<b\ne 2$, $f$ has asymptotic expansions which do not,
in general converge.  They are however generalized Borel summable.

Define $d$ by \begin{equation}
  \label{eq:eqdu}
  \frac{1}{b}+\frac{1}{d}=1
\end{equation}

 \begin{Theorem}\label{PP3}{ Let $g=j^b;\,b>1$.  Then,

(i)   The asymptotic series
of $f(z)$ for small $z$,
\begin{equation}
  \label{eq:tr1}
  \tilde{f}_0=\Gamma\left(1+\dfrac{1}{{b}}\right)z^{-\frac{1}{{b}}}-\dfrac{1}{2}+\frac{i}{2\pi}\sum_{j=1}^{\infty}(1-(-1)^{j{b}})\zeta(j{b}+1)b_{j}z^{j}
\end{equation}
is Borel summable  in $X=z^{-1/(b-1)}$, 
   along any ray $\arg(X)=c$ if $c\ne-\arg k^qs_{\pm}, k\in\NN $,   where
$s_{\pm}=t_{\pm}\mp 2\pi i t_{\pm}^{1/b}$ and   $t_{\pm}=( \pm  2\pi i/b)^{b/(b-1)}$. More precisely, (a)
 \begin{equation}
    \label{eq:eqfz}
   f(z)=\Gamma(1+1/b)z^{-1/b}-\frac{1}{2}+ z^{-b/(b-1)}\int_0^{\infty}e^{-z^{-1/(b-1)s}}H(s)ds
  \end{equation}
  where $H(s)=H_{b}(s^{b-1})$, where $H_{b}$ is analytic at zero and $H_{b}(0)=0$; (b) $H$ is
analytic on the Riemann surface of the log, with square root branch points
at all points of the form $k^qs_{\pm}, k\in\NN$ and  (c) making appropriate cuts (or working on Riemann surfaces), $u^{-(b-1)^2/b}H$ is bounded at infinity.

If $\arg(X)\in(\theta_-,\theta_+)$, then $\mathcal{LB}\tilde{f}_0=f$ \footnote{Note that 
the variable of Borel summation, or {\em critical time}, is not
$1/z$ but $z^{-1/(b-1)}$.}  In a general complex direction, $f$ has a nontrivial transseries, see (iii).

(iii) For a given direction $\varphi$,
$\sigma$ be $\pm 1$ if $\pm \varphi>0$ and $0$ otherwise. If $\sigma \arg z\in (\theta_{\sigma }, \pi/2)$, then the transseries
of $f$ is

\begin{equation}
 \tilde{f_0}(z)+\sigma \frac{i}{2\pi}\sum\limits _{k=1}^{\infty}
e^{-{s_{{\sigma}}k^{{q}}}{x^{-q+1}}}\sum\limits _{j=0}^{\infty}c_{j}(\sigma k)^{\frac{-2j{b}+2-{b}}{2({b}-1)}}z^{\frac{2j-1}{2({b}-1)}}
\end{equation}
and it is Borel summable as well.}
 \end{Theorem}

 \begin{Note}
   {\rm The duality $k^b\leftrightarrow k^d$ is the same as in van der Corput
formulas; see \cite{Montgomery}.}
 \end{Note}
\subsubsection{Examples} (i)
 For $b=3$, $f$ the transseries is given by
\begin{equation}
\tilde{f}_0(z)+\sum_{k\in\ZZ}\sigma e^{-\kappa_3k^{3/2}z^{-1/2}}\left[\left(\dfrac{\pi
i k}{6}\right)^{\frac{1}{4}}{z}^{-\frac{1}{4}}+\dfrac{5}{32}\dfrac{(ik)^{\frac{7}{4}}}{6^{\frac{3}{4}}\pi^{\frac{5}{4}}}z^{\frac{1}{4}}+...\right]
\end{equation}
where  with $\kappa_{3}=\pi^{3/2}\sqrt{\frac{32}{27}}(-1)^{\frac{1}{4}}$ and
$$ \tilde{f}_0(z)=\frac{\Gamma(4/3)}{z^{1/3}}-\dfrac{1}{2}-\dfrac{z}{120}+\frac{z^{3}}{792}+...$$

\z (ii) For $b=3/2$, with $\kappa=32 i\pi^3/27$, $f$ has the Borel summable transseries
\begin{equation}
  \label{eq:3/2f}
 \tilde{f}_0(z)+ \sigma \sum\limits _{k=1}^{\infty}e^{\kappa \sigma k^{3}z^{-2}}
\left[\dfrac{4\sqrt{2}(\sigma i)^{-\frac{1}{2}}\pi}{3}\dfrac{k^{\frac{1}{2}}}{z}+
\dfrac{(\sigma i)^{\frac{1}{2}}}{16\sqrt{2}\pi^{\frac{5}{4}}}\dfrac{z^2}{k^{\frac{5}{2}}}+...
\right]
\end{equation}
where $\theta_{\pm}=\pi/4$  and
$$\tilde{f}_0(z)=\Gamma\left(\text{\small $\frac{5}{3}$}\right)z^{-\frac{2}{3}}-\dfrac{1}{2}-\dfrac{3\zeta(\frac{5}{2})}{16\pi^{2}}z+\frac{1}{240}z^{2}+\frac{315\zeta(\frac{11}{2})}{2048\pi^{5}}z^{3}+...$$

 \subsubsection{Properties-preserving extensions beyond the barrier}
It is natural to require for an extension beyond the barrier that it
has the following properties:
\begin{enumerate}
\item It reduces to usual analytic continuation when the latter exists.

\item It commutes with all properties  with which analytic continuation
is compatible (principle of preservation of relations, a  vaguely
stated concept; this requirement is rather open-ended).

\end{enumerate}
 Borel summable series (more generally transseries)
     or suitable convergent representation representations allow for
     extension beyond the barrier, as follows. In the case $g(j)=j^b$
     (and, in fact in others in which $g$ has a convergent or summable
     expansion at infinity), $f(z)$ can be written, after Borel
     summation in the form (see
     \S\ref{AC1}) \begin{equation} \label{eq:eqac} z^{-{d}}\int_0^{
         \infty}e^{-z^{-1/(b-1)}
         s}H_1(s)ds=z^{\frac{1}{b-1}-d}\int_0^{ \infty}e^{-t}H_1(t
       z^{1/(b-1)})dt=:\int_0^{\infty}e^{-s}F(s,z)ds \end{equation} where $H_1$ is analytic near the
     origin and in $\CC$ except for arrays of isolated singularities
     along finitely many rays. Furthermore, $H_1$ is polynomially
     bounded at infinity. This means that the formal series is
     summable in all but finitely many directions in
     $z$.
     \begin{Definition}
       We define the Borel sum continuation of $f$ through point $z$ on the
       natural boundary, in the direction $d$, to be
       the Borel sum of the formal series of $f$ at $z$ in the direction $d$,
       if the Borel sum exists.

This extension simply amounts to analytically
       continuing $F(s,z)$ in $z$, in (\ref{eq:eqac}), for
       small $s$ and then analytically continuing
       $F(s,z)$ in $s$ for fixed $z$.
       along $\RR^+$.
     \end{Definition}

 \subsubsection{Notes} (cf. Appendix \S\ref{Fi}.) \begin{enumerate} \item The Borel sum provides a natural, properties-preserving,
       extension \cite{OCtransseries}. Borel summation commutes with all common operations
       such as addition, multiplication, differentiation.  Thus,
       the function and its extensions will  have the same
       properties.

\item
Also, when a series converges, the Borel sum coincides with the usual
sum.  Thus, when analytic continuation exists, it coincides with the
extension.

\item

 With or without a boundary, the Borel sum of a
divergent series changes as the direction of summation crosses the
{\em Stokes directions} in $\CC$.  Yet, the properties of the family of
functions thus obtained are preserved.
There may then exist extensions {\em along} the barrier as well. Of course,
all this cannot mean that there is analytic continuation across/along
the boundary.
\end{enumerate}

See also Eq. (\ref{convt}), a convergent expansion, where $z$ can also
be replaced by $-z$.  In this case, due to strong lacunarity, the extension
{\em changes} along {\em every} direction, as though there existed densely
many Stokes lines.

\subsection{Universal behavior near boundary in specific cases}
\begin{figure}
  \centering
  \includegraphics[scale=0.3]{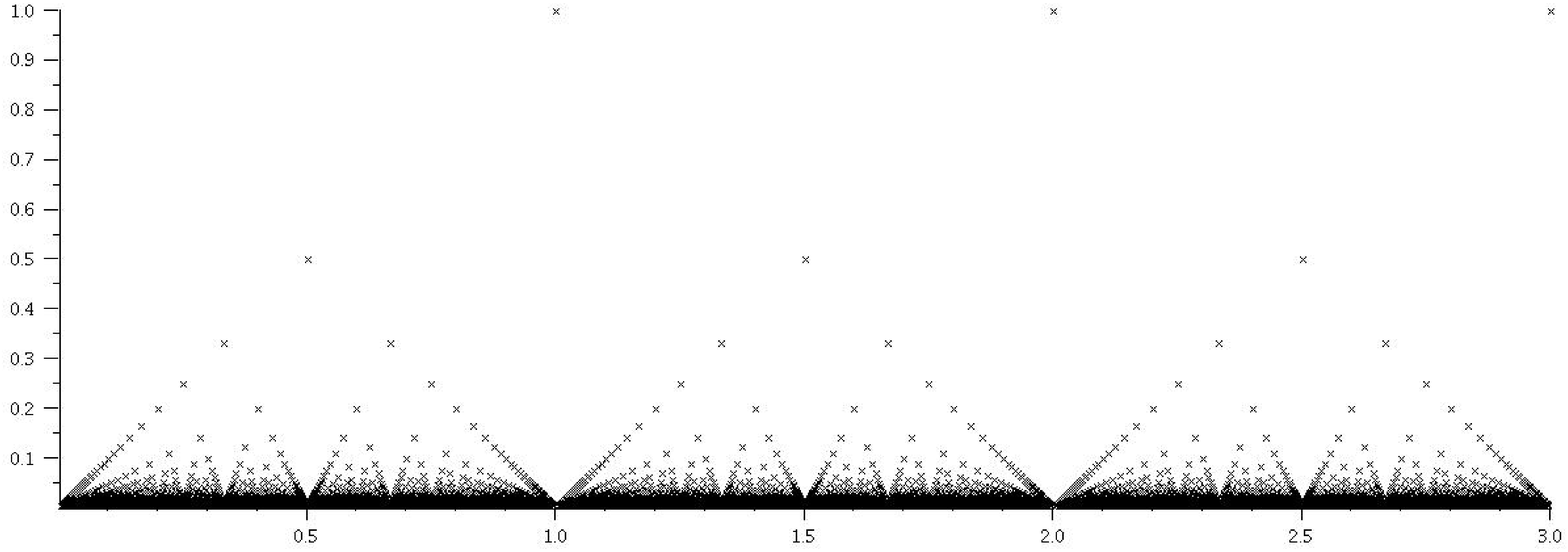}
  \caption{The standard function $Q(x)$. The points above $Q=0.05$ are
the only ones present in the actual graph.}
  \label{fig:1}

  \includegraphics[scale=0.3]{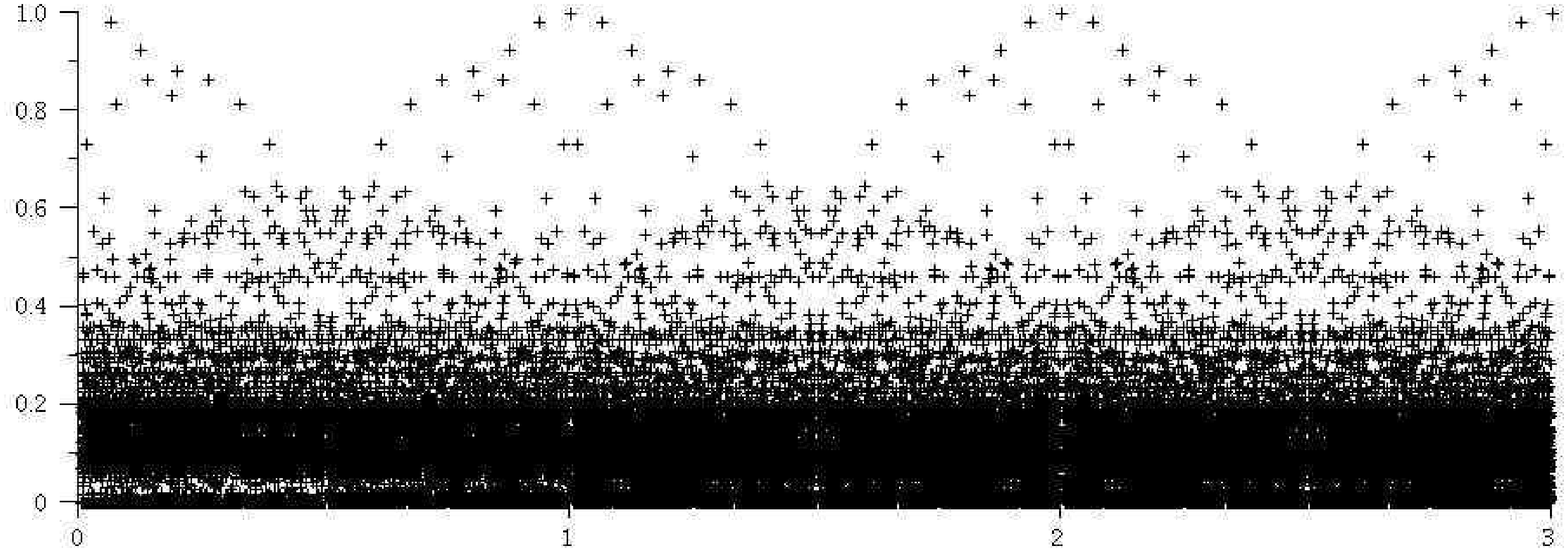}
  \caption{Point-plot graph of $Q_{4,4}$, normalized to
    one.}
  \label{fig:1}
 \end{figure}
\begin{figure}
  \includegraphics[scale=0.4]{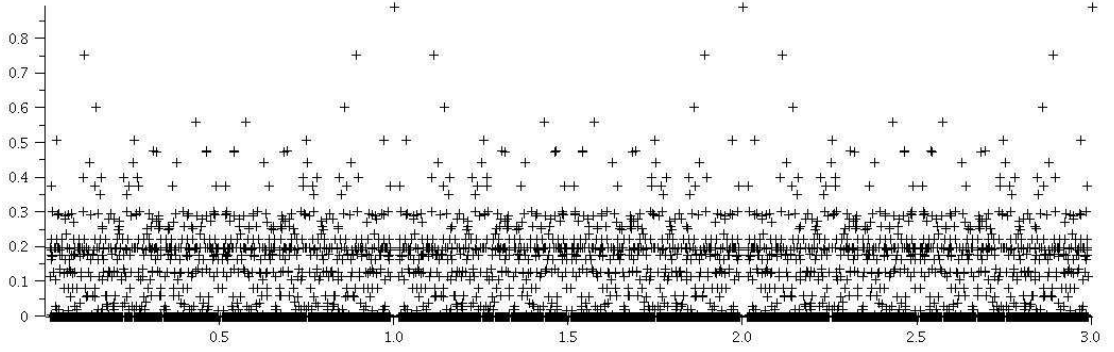}
  \caption{Point-plot graph of $Q_{3,3}$; $Q_{3,2}$ follows from it
through the transformation (\ref{eq:dual}).
}
  \label{fig:1}

\end{figure}

In many cases, $\phi(g^{-1}(u/x))$ has an asymptotic expansion, and then
$\Phi$ in turn has an asymptotic expansion in $x$.
Detailed behavior along the boundary can be obtained in special cases such as  $g(j)=j^b;\,b\in\NN$,
or $g(j)=j^{(b+1)/b};\,b\in\NN$. Properly scaled sums converge to everywhere discontinuous
functions.
\begin{Theorem}\label{Pp4}
  (i) If $g(j)=j^b$, $b\in\NN$, we have $\displaystyle \limsup
  x^{1/b} |f(x +iy)|<M<\infty$. We let
  $Q_{b,d}(s)=\displaystyle \lim_{x\to
    0^+}x^{1/d}f(x+2\pi i s)$ (whenever it exists); then,

\begin{multline}
  \label{eq:bd2}
\frac{Q_{b,b}(s)}{\Gamma(1+b^{-1})}=\begin{cases} \displaystyle  \,\,\frac{1}{n}\sum_{l=1}^{n}e^{-2\pi
i\dfrac{m}{n}l^{b}}, \ s=\frac{m}{n}, \ m,n\in \NN\ \text{relatively prime} \\ \\
0\ \ \ a.e. \end{cases}
\end{multline}

(ii) For $b=3/2$ and $z\to 0^+$, $f$ grows like $z^{-2/3}$. For any other point
on $i\RR$ the growth is slower, at most $(\Re\, z)^{-1/2}$. Furthermore,

\begin{equation}\label{case3/2}
  \frac{Q_{3/2,1/2}(\sqrt{s})}{\sqrt{6\pi i}s^{1/4}}=\begin{cases}\displaystyle\frac{1}{n}\sum_{l=1}^{n}e^{-2\pi
i\dfrac{m}{n}l^{3}},\ \ s =\dfrac{16}{27}\frac{n}{m}, \ n,m\text{ as in (i)}\\ \\
 0\ \ a.e.\end{cases}
\end{equation}
In particular, we have the profile duality relation
\begin{equation}
  \label{eq:dual}
  Q_{3/2,1/2}(s)=\frac{\sqrt{6\pi i}s^{1/2}}{\Gamma(4/3)}Q_{3,3}\(2^83^{-6}s^{-2}\)
\end{equation}
for any $s\in\RR$ for which $Q_{3/2}(s)$ and/or
$Q_{3,3}(s)$ is well defined, for instance, in the cases given in
(\ref{eq:bd2}) and (\ref{case3/2}).
\end{Theorem}
For large $n$, the sum over $l$ in Eq. (\ref{eq:bd2}) is,
statistically, expected to be of order one.  After $x^{\frac{1}{d}}$
rescaling, the template behavior ``in the bulk'' is given by the
familiar function $Q(x)=n^{-1}$ if $x=m/n, (m,n)\in\NN^2$ and zero
otherwise, shown in Fig.  \ref{fig:1}.
\subsection{Fourier series of the B\"otcher map and structure of Julia sets}\label{holom}

We show how lacunary series are building blocks for fractal structures
appearing in holomorphic dynamics (of the vast literature on  holomorphic
dynamics we refer here in particular to \cite{Beardon}, \cite{Devaney} and \cite{Milnor}).
In \S\ref{S5} we mention a few known facts about polynomial
maps. Consider for simplicity the
quadratic map
\begin{equation}
  \label{eq:logist1}
  x_{n+1}=\lambda x_n(1-x_n)
\end{equation}
It will be apparent from the proof that the results and method extend easily
to  polynomial iterations of the form
\begin{equation}
  \label{eq:eqpol1}
  x_{n+1}=\lambda P_k(x_n)
\end{equation}
with $\lambda$ relatively small.
The substitution $x=-(\lambda y)^{-1}$ transforms (\ref{eq:logist1}) into
\begin{equation}
  \label{eq:infty1}
  y_{n+1}=\frac{y_n^2}{\lambda(1+ y_n)}=f(y_n)
\end{equation}

By B\"otcher's theorem (we give a self contained proof in
\S\ref{S5}, for (\ref{eq:logist1}), which extends in fact to the
general case), there exists a map $\phi$,\footnote{The notations
$\phi,\psi, H$, designate different objects than those in previous
section.} analytic near zero, with $\phi(0)=0$,
$\phi'(0)=\lambda^{-1}$ so that $(\phi f \phi^{-1})(z)=z^2$. Its
inverse, $\psi$, conjugates (\ref{eq:infty1}) to the canonical map
$z_{n+1}=z_n^2$, and it can be checked that
\begin{equation}
  \label{eq:eqG}
  \psi(z)^2=\lambda \psi(z^2)(1+\psi(z));\ \ \psi(0)=0,\ \psi'(0)=\lambda
\end{equation}
Let $\mathcal A(\mathbb{D})$ denote the Banach space of functions
analytic in the unit disk $\mathbb{D}$ and continuous in  $\overline{\mathbb{D}}$, with the sup norm. We define  the linear operator $\mathfrak T=\mathfrak T_2$, on $\mathcal A(\mathbb{D})$ by
\begin{equation}
  \label{eq:defT}
  (\mathfrak T f)(z)=\frac{1}{2}\sum_{k=0}^{\infty}2^{-k}f(z^{2^k})
\end{equation}
This is the inverse of the operator $f\mapsto 2f-f^{\vee 2}$,
where $f^{\vee p}(z)=f(z^p)$.
Clearly, $\mathfrak T f$ is an isometry on   $\mathcal A(\mathbb{D})$
and it maps simple functions, such as generic polynomials, to
functions having $\partial \mathbb{D}$ as a natural boundary; it reproduces $f$
across vanishingly small scales.

\bigskip

\begin{Theorem}\label{T1}

(i) $\psi$ and $H=1/\psi$ are analytic in ($\lambda,z$) in $\mathbb{D}\times\mathbb{D}$ ($\lambda\in \mathbb{D}$ corresponds to the main cardioid
in the Mandelbrot set, see \S\ref{S5}), and continuous in $\mathbb{D}\times\overline{\mathbb{D}}$. The series
\begin{equation}
  \label{eq:formG}
\psi(\lambda,z)=\lambda z+\sum_{k=1}^{\infty}z\lambda^{k+1} \psi_k(z)
\end{equation}
converges in $\mathbb{D}\times\overline{\mathbb{D}}$ (and so does the series
of $H$), but not in  $\overline{\mathbb{D}}\times\overline{\mathbb{D}}$. Here
\begin{equation}
  \label{eq:formPhi1}
 \psi_1(z)=\mathfrak T z=\sum_{k=1}^{\infty}2^{-k+1}z^{2^k}
\end{equation}
(note that $\psi(z)=\frac{1}{2}z+\frac{1}{2}\int_0^zs^{-1}h(s)ds$ with $h$ given in (\ref{eq03})),  and in general
\begin{equation}
  \label{eq:frakt}
 \psi_k =\mathfrak{T}\left(z\sum_{j=0}^{k-1}\psi_j^{\vee 2}\psi_{k-1-j}-\sum_{j=1}^{k-1}\psi_j\psi_{k-j}\right)
\end{equation}

(ii) All $\psi_k, k\ge 1$ have binarily lacunary series: In  $\psi_k$,
the coefficient of $z^p$ is nonzero only if $p$ has at most $k$ binary-digits equal to $1$, i.e.,
\begin{equation}
  \label{eq:pk}
  p=2^{j_1}+\cdots+2^{j_k},\ \ j_i=0,1,2,... \ (\text{``$k$''is the same as in $\psi_k$})
\end{equation}
(iii) For $|\lambda|<1$, the Julia curve of (\ref{eq:logist1}) is given
by a uniformly convergent Fourier series, by (i),
\begin{equation}
  \label{eq:eqJ}
  J=\left\{-\left(\Re\, H(e^{it}),\Im\, H(e^{it})\right): t\in [0,2\pi)\right\}
\end{equation}
\end{Theorem}
\begin{Remark}{\rm
The effective lacunarity of the Fourier series makes calculations of the
Julia set numerically effective if $\lambda$ is not too large.}
\end{Remark}

\begin{figure}
  \centering
 \includegraphics[scale=0.3]{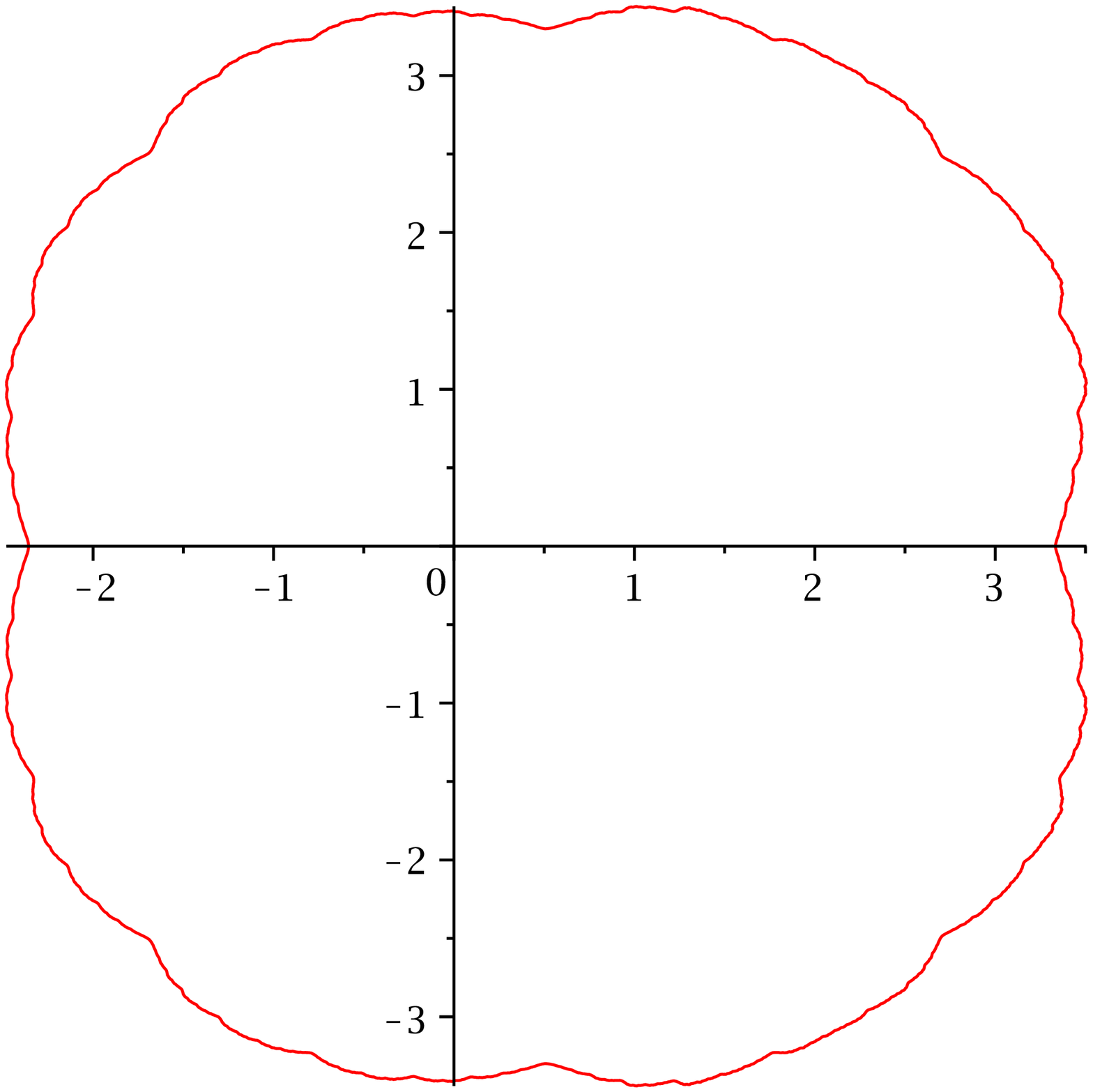}\includegraphics[scale=0.3]{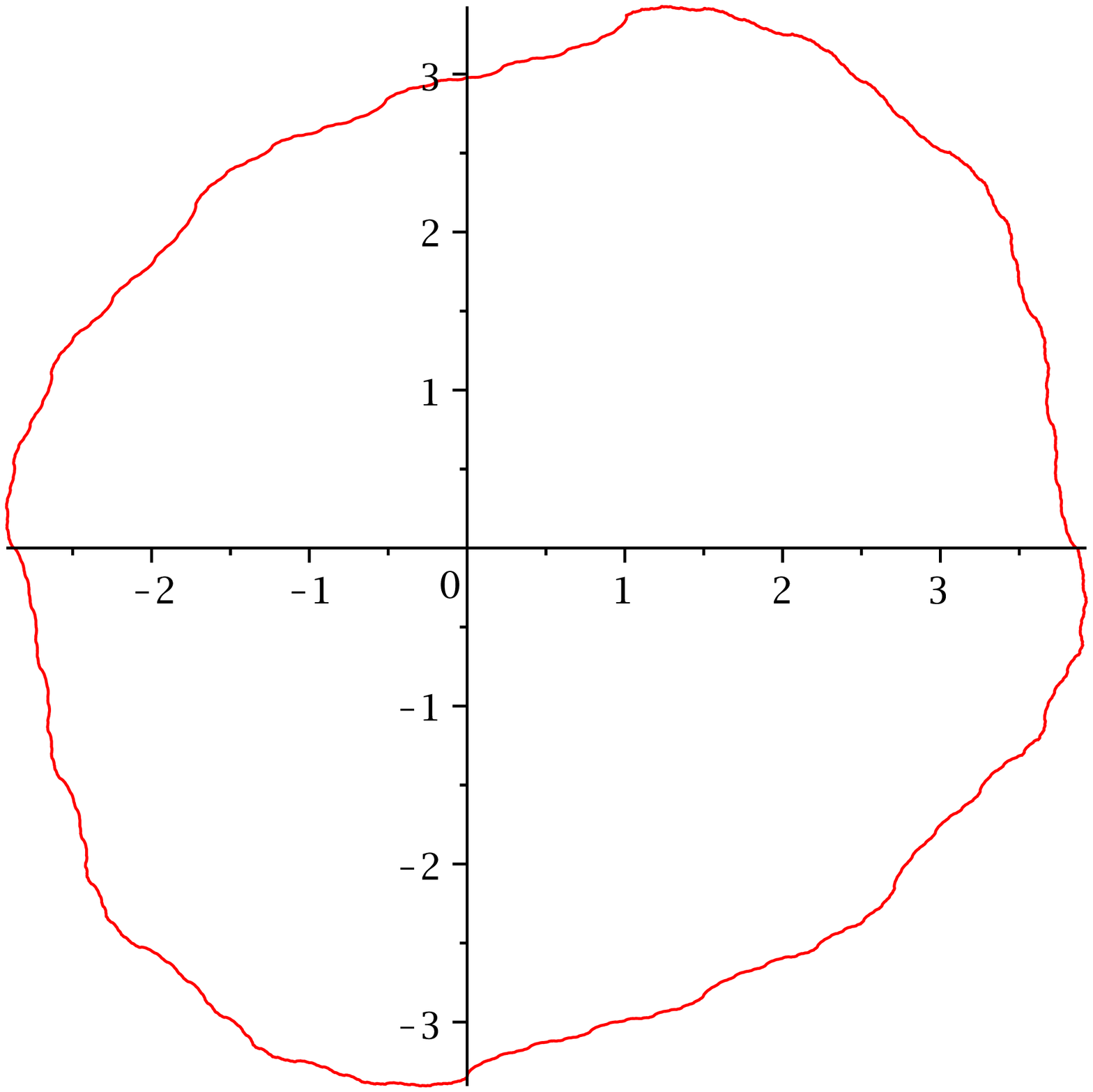}
  \caption{The Julia set of $x_{n+1}=\lambda x_n(1-x_n)$, for $\lambda=0.3$ and
 $\lambda=0.3i$,
calculated from the Fourier series  (\ref{eq:formG}) discarding all $o(\lambda^2)$  terms. They coincide, within
plot precision, with numerically calculated ones using standard iteration of maps
algorithms.
}
  \label{fig:3}

\end{figure}

\begin{Note} \label{N4} $ $ \hskip -1cm {\rm
    \begin{enumerate}
    \item Lacunarity of $\psi_k$ is a strong
    indication that $\psi$ has a natural boundary (in fact, it does have one),
    but not a proof.  Transseriation at {\em singular} points and
    summation of convergent series (here in the parameter $\lambda$)
    {\bf do not commute}.  The transseries of $H$ on the barrier can
    be calculated by transasymptotic matching, see \cite{Inventiones},
    but in this case is more simply found directly from the functional
    relation; see Note \ref{Explc}.
\item In assessing the fine structure
 of the fractal using the Fourier expansion truncated to $o(\lambda^n)$, the scale of analysis
cannot evidently  go below $O(\lambda^n)$.
    \end{enumerate}
}
\end{Note}

\begin{Remark}{\rm

    \begin{enumerate}

 \item For $|\lambda|$ sufficiently small,
Theorem~\ref{T1} provides a convenient way to determine
the Julia set as well as the discrete evolution on the boundary.

 \item  For small $\lambda$, the self similar structure
is seen in
$$\psi_1(\rho \exp(2\pi i m/2^n))=
\sum_{k=1}^{n-1}\frac{ \rho^{2^j}
\exp(2\pi i m 2^{k-n})}{2^j}+\frac{ \rho^{2^n-2}}{2^{n-1}}\psi_1(\rho)$$
where the sum is a polynomial, thus analytic. Up to a scale factor
of $2^{-n+1}$,  if $|\lambda|\ll 2^{-n+1}$,  the nontrivial structure of $\psi$
at $\exp(2\pi i m/2^n)$ and at $1$ are the same, see Note \ref{N4}; that is
$$\psi(\exp(2\pi i m/2^n))={2^{-n+1}}\psi(1)+\text{regular}+o(\lambda)$$
Exact transseries can be obtained for $\psi$; see  also Note~\ref {Explc}.

\item For iterations  of the form
$x_{n+1}=\lambda^{k-1}P_{k}(x)$ where $P_k$ is a polynomial of degree $k>2$,  the calculations and the results, for small
$\lambda$, are essentially the same. The lacunary series
would involve the powers $z^{k^j}$. For instance if the
recurrence is $x_{n+1}=\lambda^2x_n(1-3x_n+x_n^2)$, then $\psi$
is to be replaced by the solution of
$$\psi=\lambda^2\psi^{\vee 3}(1+3\psi+\psi^2)$$
and the small $\lambda$ series will now have
$\psi_1=\sum_{k=1}^{\infty}3^{-k-1}{{z^{3^k}}}$ and so on.
    \end{enumerate}
  }
\end{Remark}

\section{Proofs}
\begin{Lemma}\label{linv}
 Under the assumptions of Theorem~\ref{P1}, (iii), $h(y):=g^{-1}(y)$ also has a
  differentiable asymptotic power series as $y\to\infty$.
  \end{Lemma}
\begin{proof}
    Straightforward inversion of power series asymptotics, cf \S\ref{appendix}.
  \end{proof}

\subsection{Proof of Theorem \ref{P1}}
The proof of (\ref{eq:asbh}) (i)  essentially amounts to showing that
$|f|^2$ is diagonally dominant, in that terms containing $g(j)$ and $g(k)$
with $j\ne k$ are comparatively small, as shown in \S\ref{genblowup}.

(ii) Since $g$ is increasing on
$(0,\infty)$, the result follows from  the usual
integral upper and lower bounds for a sum. Equation (\ref{eq:asympt1})
follows from simple calculations, cf.  \S\ref{appendix}.

\subsection{Proof of  Theorem \ref{Prop2}}
We prove part (ii); part (i) is similar, and simpler. By standard Fourier analysis we get
\begin{equation}
  \label{eq:eq21}
\{u\}=\frac{1}{2}-\sum_{k=1}^{\infty}\frac{\sin2k\pi u}{k\pi}, \ \ \ \ u\notin\mathbb{Z}\end{equation}
where
\begin{equation}
  \label{eq:eqM}
  \left\|\{u\}-\frac{1}{2}+\sum_{k=1}^{M}\frac{\sin2k\pi u}{k\pi}\right\|_{\infty}\le C_M \to {\text{\rm  Si}}(\pi) \text{ as } M\to \infty
\end{equation}
where Si is the sine integral, and Si($\pi$) is the constant in the Gibbs phenomenon.
      Let $g_m$ be an analytic function such that
    $g(s)-g_m(s)=o(s^{-m})$ for large $s$.
    \begin{Lemma}
      The analysis reduces to the case where $h$ is a finite sum of powers. Indeed,
\begin{equation}
  \label{eq:eqf}
  f=e^{-xg(0)}+\int_0^{\infty} e^{-xu}g^{-1}(u)du+\int_{g(0)}^{\infty}e^{-xu} \{h_m(u)\}du+R_{m-1}(x)
\end{equation}
where $R_{m-1}$ is $C^{m-1}$ and $h_m$ is a truncation of the asymptotic
expansion of $h$, such that $h(u)-h_m(u)=o(u^{-m})$.
    \end{Lemma}
    \begin{proof}
We have

\begin{multline}
  \label{eq:sum2}
  \int_0^{\infty} e^{-xu} \{h(u)\}  du=
\sum_{N=0}^{\infty}\int_{g(N)}^{g(N+1)} e^{-xu}(h(u)-N) du\\=
\sum_{N=0}^{\infty}\int_{g_m(N)}^{g_m(N+1)} e^{-xu}(h_m(u)-N)du
+R_{m-1}(x)
\end{multline}
where
\begin{multline}
  \label{eq:eqR}
R_{m-1}(x)=\sum_{N=0}^{\infty}\left(\int_{g(N)}^{g_m(N)}+\int_{g_m(N+1)}^{g(N+1)}\right)
 e^{-xu} \{h(u)\}  du\\+\sum_{N=0}^{\infty}\int_{g_m(N)}^{g_m(N+1)}e^{-xu}(h(u)-h_m(u))du
\end{multline}
Using (\ref{eq:eq21}) and (\ref{eq:eqM}) the proof follows and the
fact that $g(N)-g_m(N)=o(N^{-m})$ and $h(u)-h_m(u)=o(u^{-m})$, the sum is rapidly convergent, and the result follows. \end{proof}
\begin{Lemma}
  If $h$ is a finite sum of powers, then $f-e^{-xg(0)}-\int_0^{\infty} e^{-xu}g^{-1}(u)du\in C^\infty$.
\end{Lemma}
\begin{proof}

Using (\ref{eq:eq21}) and (\ref{eq:eqM}) we have
\begin{multline}
  \label{eq:eqhm}
  f_1(x)=  x\int_0^{\infty}e^{-xu} \{h_m(u)\}du=\frac{1}{2}-
\sum_{k=1}^{\infty}\int_0^{\infty}e^{-xu}\frac{\sin 2k\pi h_m(u)}{k\pi}du\\
=\frac{1}{2}-
\sum_{k=1}^{\infty}\frac{1}{2k\pi i}\sum_{\sigma=\pm 1}\sigma  \int_0^{\infty}e^{-xu+\sigma 2k\pi i h_m(u)}du
\end{multline}
  We deform the contours of integration along the directions
$\sigma\alpha$ respectively, say for $\alpha=\pi/2$.  The integral
\begin{equation}
  \label{eq:vert}
 f_1= \int_0^{i\infty}u^n e^{2k\pi i h_m(u)}du
\end{equation}
exists for any $n$ and it is estimated by
\begin{equation}
  \label{eq:estm3}
 \left|\int_0^{i\infty}u^n e^{2k\pi i h_m(u)}du\right|<const.  k^{-b(n+1)}\Gamma(n(b+1))
\end{equation}
The termwise nth derivatives at  $0^+$ of the series of $f_1(x)$ converge rapidly,
and the result follows.
\end{proof}

\subsection{Proof of Theorem \ref{Cexp}}
\begin{proof}
  Let $\epsilon>0$ be arbitrary, let
$$S_N=\sum_{j=1}^N e^{i y g(j)};\ \ \ S_0:=0$$
and let $N_1$ be large enough so that $|S_N/\rho(N)-L|\le \epsilon$ for $N>N_1$. Then, by looking at $e^{i\phi} f$ if necessary, we can assume that $L\ge 0$. We have
\begin{multline}
  \label{eq:eqf3}
  f(x+iy)=\sum_{k=1}^{\infty}e^{-xg(k)}(S_{k}-S_{k-1})=\sum_{k=1}^{\infty}\frac{S_k}{\rho(k)}(e^{-xg(k)}-e^{-xg(k+1)})\rho(k)\\=L(y)\sum_{k=1}^{\infty}(e^{-xg(k-1)}-e^{-xg(k)}){\rho(k)}+\sum_{k=1}^{\infty}(e^{-xg(k)}-e^{-xg(k+1)})d_k{\rho(k)}
\end{multline}
where $d_k=L(y)-S_k/\rho(k) \to0$ as $k\to\infty$. Now, 
\begin{equation}
  \label{eq:67}
 \sum_{k=1}^{\infty}e^{-xg(k)}(\rho(k+1)-\rho(k)) \sim \int_0^{\infty}e^{-xg(k)}\rho'(k)dk =:\Phi_{\rho}(x).
\end{equation}
and under the given assumptions $\Phi_{\rho}\to \infty$ as $x\to 0$. 
Note that $\Phi_{\rho}\to \infty$ as $x\to 0$ and the facts that $e^{-xg(k)}-e^{-xg(k+1)}>0$, $\rho(k)>0$ and $d_k\to 0$, readily imply that the last sum in (\ref{eq:eqf3}) is $o(\Phi_{\rho}(x))$; (\ref{eq:lsup}) is follows in a similar way.
\end{proof}

\subsection{Proof of Theorem\ref{Pp4}}
We rely on Theorem \ref{Cexp}, and analyze the case $b=3/2$; the
case $b\in\NN$ is simpler. We have $\beta=t/(2\pi)$. Let
$\check{S}_j=\sum \limits_{k=1}^{j}e^{-2\pi i \dfrac{m}{n} k^3}\!\!.$ It
is clear that for $m,n\in\NN$ we have $j^{-1}\check{S}_j\to \sum
\limits_{k=1}^{n}e^{-2\pi i \dfrac{m}{n} k^3}=L_{mn}$. On the other
hand, by summation by parts we get

\begin{multline}
  \label{eq:eq01}
S_N= \frac{2\sqrt{2}e^{\pi i/4}}{3\beta}\sum_{k=1}^{k_N-1} k^{1/2}
 \exp\left(-\frac{8\pi i}{27\beta^2}k^3\right)+O(N^{1/4})=O(N^{1/4})\\+
\frac{2\sqrt{2}e^{\pi i/4}}{3\beta}(k_N-1)^{3/2}\frac{\check{S}_{k_N-1}}{k_N-1}
+\frac{2\sqrt{2}e^{\pi i/4}}{3\beta}\sum_{k=1}^{k_N-1}k^{-1}\check{S}_k k(k^{1/2}-(k-1)^{1/2})
\end{multline}
and it is easy to check that, as $N\to\infty$ we have
\begin{equation}
  \label{eq:sum}
S_N\sim  \left(\frac{2\sqrt{2}e^{\pi i/4}}{3\beta}(1+1/3)\right) k_N^{3/2} L_{mn}=\frac{8\sqrt{2}e^{\pi i/4}}{9\beta}\sqrt{\frac{3\beta}{2}}N^{3/4}
\end{equation}
where we used the definition of $k_N$ following
eq. (\ref{eq:eq01}), which implies $k_N\sim \frac{3}{2}N^{1/2}\beta$. The result follows by changes of variables, using
(\ref{eq:lim2}) and noting that
\begin{equation}
  \label{eq:int4}
 \frac{3}{4} \int_0^{\infty}e^{-x k^{3/2}}k^{-1/4}dk=\frac{\sqrt{\pi}}{2\sqrt{x}}
\end{equation}

It is also clear that $|j^{-1}\check{S}_j|\le 1$ and a similar calculation
provides an overall upper bound.

Section \ref{Direct,3} provides an independent way to calculate the behavior
along the boundary.
\subsection{Proof of Proposition \ref{b=2}}
For $b=2$ using  (\ref{eq:eqhm}) and
$\int_{0}^{\infty}e^{-px}\sin2k\pi
p^{\frac{1}{2}}dp=\pi^{3/2}ke^{-\frac{k^{2}\pi^{2}}{x}}x^{-3/2}$
we immediately obtain

\[
f(x)=\frac{\sqrt{\pi}}{2\sqrt{x}}-\dfrac{1}{2}+\frac{\sqrt{\pi}}{\sqrt{x}}\sum_{k=1}^{\infty}e^{-\frac{k^{2}\pi^{2}}{x}}\]

\subsection{The case $g(j)=a^j$}

It is easy to see that, for this choice of $g$, we have the functional
relation $$f(z)-f(az)=e^{-z}$$

Since $e^{-z}\to 1$ as $z\rightarrow0+$, the leading behavior formally
satisfies $f(z)-f(az)\sim1$ i.e. $f(z)\sim-\dfrac{\log z}{\log a}$

In view of  this we let $f(z)=-\dfrac{\log z}{\log a}+G(z)$ which
gives  $G(z)-G(az)=e^{-z}-1$ {\em i.e.}

\begin{equation}
  \label{eq:eqg}
 G(z)=G(z/a)+(1-e^{-z/a})
\end{equation}
We first obtain a solution a solution $\check{G}$ of the homogeneous equation and then
write $G=\check{G}+h(a^z)$, where $h$ now satisfies
\begin{equation}
  \label{eq:eqhy}
  {h}(y)={h}(y+1)
\end{equation}
 Iterating (\ref{eq:eqg}) we obtain
\[
\check{G}(z)=\sum_{n=0}^{\infty}(1-e^{-z/a^n})=-\sum_{n=0}^{\infty}\sum_{k=1}^{\infty}\frac{(-z)^{k}}{k!a^{nk}}=\sum_{n=1}^{\infty}\frac{(-z)^{n}}{n!(1-a^{n})}\]
which is indeed an entire function and satisfies the functional
relation.  Now we return to (\ref{eq:eqhy}). Since by its connection to
$f(z)$, ${h}$ is obviously smooth in $y$, it can be expressed in
terms of its Fourier series

\[
{h}(y)=\sum_{k=-\infty}^{\infty}c_{k}e^{2k\pi iy}\] where the
coefficients $c_{k}=\bar{c}_{-k}$ can be found by using the original
function $f$. Recall that we have $$f(z)=-\dfrac{\log z}{\log
  a}+\check{G}(z)+{h}\(\dfrac{\log z}{\log a}\)$$ which
implies

\[
{h}(y)=f(a^{y})+y-\check{G}(a^{y})\]
\begin{Lemma}\label{Lgeom}
  The Fourier coefficients of the periodic function ${h}$ are given
by
\begin{equation}
  \label{eq:ck}
  c_{k}=\frac{1}{\log a}\Gamma\(-\frac{2k\pi i}{\log a}\)\:\(k\neq0\)
\end{equation}
and
$$c_{0}=\int_{0}^{1}f(a^{y})dy+\int_{0}^{1}ydy-\int_{0}^{1}\check{G}(a^{y})dy$$
Since $c_{k}\sim\sqrt{\dfrac{1}{\pi\log a}}e^{\frac{-|k|\pi^{2}}{\log a}}$,
(implying that the Fourier  expansion  for $h$ is valid exactly
for $\Re\(z\)>0$) we have
\[
f\(z\)=-\dfrac{\log z}{\log a}+\sum_{n=1}^{\infty}\frac{\(-z\)^{n}}{n!\(1-a^{n}\)}+c_{0}+\frac{1}{\log a}\sum_{k\neq0}\Gamma\(-\frac{2k\pi i}{\log a}\)z^{\frac{2k\pi i}{\log a}}\]
The series further resums to
\begin{equation}
  \label{eq:resum}
  f\(z\)=-\dfrac{\log z}{\log
a}+\sum_{n=1}^{\infty}\frac{\(-z\)^{n}}{n!\(1-a^{n}\)}+c_{0}-\frac{z}{2\pi
i}\int_{0}^{\infty}\log_{\RR}\(-s^{\frac{2\pi i}{\log a}}\)e^{-sz}ds
\end{equation}
valid for $z>0$.
\end{Lemma}
The proof is given in \S\ref{pfLgeom}.

Similar results hold for other rational angles if $\bb\in\NN$ (by
grouping terms with the same phase). For example,
\begin{multline}
  \sum_{n=0}^{\infty}e^{-2^{n}(z+\frac{2}{5}\pi i)}\\=e^{-\frac{2}{5}\pi
i}\sum_{n=0}^{\infty}e^{-2^{4n}z}+e^{-\frac{4}{5}\pi
i}\sum_{n=0}^{\infty}e^{-2^{4n+1}z}+e^{-\frac{8}{5}\pi
i}\sum_{n=0}^{\infty}e^{-2^{4n+2}z}+e^{-\frac{6}{5}\pi
i}\sum_{n=0}^{\infty}e^{-2^{4n+3}z}\\
=e^{-\frac{2}{5}\pi
i}\sum_{n=0}^{\infty}e^{-16^{n}z}+e^{-\frac{4}{5}\pi
i}\sum_{n=0}^{\infty}e^{-16^{n}(2z)}+e^{-\frac{8}{5}\pi
i}\sum_{n=0}^{\infty}e^{-16^{n}(4z)}+e^{-\frac{6}{5}\pi
i}\sum_{n=0}^{\infty}e^{-16^{n}(8z)}
\end{multline}
\subsection{Proof of Theorem~\ref{PP3}}\label{AC1}
(i) and (ii):
 By an argument similar to the one leading to (\ref{eq:eqhm}) we have
\begin{equation}
  \label{eq:dec3}
  f(z)=\int_0^{\infty} e^{-zg(s)}  ds-\frac{1}{2}+
\sum_{k=1}^{\infty}\frac{1}{2k\pi i}\sum_{\sigma= \pm 1}\int_0^{\sigma i\infty}e^{-zu+\sigma 2k\pi i u^{1/b}}du
\end{equation}
The exponential term ensures convergence in $k$. 
Taking $s^b=t$ we see that
\begin{equation}
  \label{eq:eqt1}
  \int_0^{\infty} e^{-z s^b}dp=\Gamma(1+1/b)z^{-1/b}
\end{equation}
We now analyze the case $\sigma=-1$, the other case being similar. A term in the sum is
\begin{multline}
  \label{eq:trm1}
  \frac{1}{2k\pi i}\int_0^{ i\infty}e^{-zu-2k\pi i u^{1/b}}du=
  \left(\frac{k}{z}\right)^{{d}}\int_0^{
    \infty}e^{-\nu_k(t+2\pi i t^{1/b})}dt\\=\left(\frac{k}{z}\right)^{{d}}
\int_0^{
    \infty}e^{-\nu_k s}{\check{H}_{-}}(s)ds=z^{-{d}}\int_0^{
    \infty}e^{-z^{-1/(b-1)} s}{\check{H}_{-}}(s/k^{{d}})ds
\end{multline}
where $\nu_k={k^{{d}}}{z^{-1/(b-1)}}$, ${\check{H}_{-}}(s)=\phi'(s)$ and 
\begin{equation}
  \label{eq:eqphi}
 \Phi(\phi(u)):= \phi(u)+2\pi i \phi(u)^{1/b}=u
\end{equation}
(a) Near the origin, we write $\phi=u^b H_{b}$, and we get 
$$H_{b}=\frac{1}{2\pi i}(1-u^{b-1}H_{b})^b$$
and analyticity of $H_{b}$ in $u^{b-1}$ follows, for instance, from the contractive
mapping principle. 

(b) The only singularities of $\phi$, thus of ${\check{H}_{-}}$, are the points implicitly given by $\Phi'=0$.

(b) It is also easy to check that for some $C>0$ we have
\begin{equation}
  \label{eq:eqH1}
  |{\check{H}_{-}}(u)|\le C\frac{|u^{b-1}|}{1+|u^{b-1}|}
\end{equation}
uniformly in $\CC$, with a
cut at the singularity, or on a corresponding Riemann surface.
From (\ref{eq:eqH1}) we see that the sum
\begin{equation}
  \label{eq:eqsum}
  H_{-}=\sum_{k=1}^{\infty} {\check{H}_{-}}(s/k^{{q}})
\end{equation}
converges, on compact sets in $s$, at least as fast as $const\sum k^{-b}$,
thus it is an analytic function wherever {\em all} ${\check{H}_{-}}(s/k^{{d}})$ are
analytic, that is, in $\CC$ except for the 
points $k^d s_0$. Using an integral estimate, we get the
global bound
\begin{equation}
  \label{eq:eqinfty}
  |H_{-}|\le const \sum_{k=1}^{\infty} \frac{|u|^{b-1}}{k^q+|u|^{b-1}}\le const |u|^{(b-1)^{2}/b}
\end{equation}
as $|u|\to \infty$.

The function $H$ in the lemma is simply $H_{-}+H_{-}$ where 
$H_+$ is obtained from $H_{-}$ by replacing $i$ by $-i$. The calculation
of the explicit power series is straightforward, from (\ref{eq:eqphi}),
(\ref{eq:eqsum}) and the similar formulae for $H_+$, using dominated 
convergence based on (\ref{eq:eqinfty}). We provide the details for convenience.

We  write the last term in (\ref{eq:trm1}) in the form 
\begin{equation}
  \label{eq:deffk}
  f_{k}(z)=\(\frac{k}{z}\)^{\frac{b}{b-1}}\intop_{C}e^{-s(\frac{k^{b}}{z})^{\frac{1}{b-1}}}\check{H}_{-}ds
\end{equation}
where  the contour $C$ is a curve from the origin to $\infty$ in the
first quadrant.

Watson's lemma implies \[
f_{k}(z)=\sum_{j=1}^{N}b_{j}k^{-jb}z^{j-1}+R_{N}(k,z)\] where
$b_{j}$can be calculated explicitly from Lagrange-B\"urmann inversion
formula used for the inverse function $\phi^{-1}$ and $R_{N}(k,z)\leqq
Ck^{-(N+1)b}z^{N}$, for arbitrary $N\in\mathbb{N}$. It follows that
\begin{equation}
  \label{eq:srjtob}
  f(z)\sim\Gamma(1+\dfrac{1}{b})z^{-\frac{1}{b}}-\dfrac{1}{2}+\frac{i}{2\pi}\sum_{j=1}^{\infty}(1-(-1)^{jb})\zeta(jb+1)b_{j}z^{j}
\end{equation}
which holds  for $z\rightarrow0$ in the right half plane in any
direction not tangential to the imaginary axis. 

(iii) We  obtain the transseries (which gives us  information
near the imaginary axis) by using the global properties of
$\check{H}_{-}$, and standard deformation of the Laplace contour.

As $z$ goes around the complex plane, as usual in Laplace-like
integrals, we rotate $s$ in (\ref{eq:deffk}) simultaneously,
to keep the exponent real and positive. In the process,
as we cross singularities, we collect a 
contribution to the integral from the point $s_{-}$ above; the contribution
is an integral around a cut originating at $s_{-}$. The singularity is integrable, and
collapsing the contour to the cut itself, we get a 
contribution again in the form of a Laplace transform. 
This is the Borel sum (in the same variable, $z^{-1/(b-1)}$. 

Generically $s_{-}$ is a
square root branch point and we have
\[
\frac{ds_{-}}{ds}=\sum_{j=0}^{\infty}\tilde{c}_{j}\left[s-\left(\dfrac{b
i}{2\pi}\right)^{\frac{b}{1-b}}\right]^{\frac{j-1}{2}}\]
The asymptotic expansion of the cut contribution, Borel summable
as we mentioned, is
\[
\exp \left({-(2\pi)^{\frac{b}{b-1}}(b^{\frac{b}{1-b}}-b^{\frac{1}{1-b}})i^{\frac{b}{1-b}}\(\frac{k^{b}}{z}\)^{\frac{1}{b-1}}}\right)\sum_{j=0}^{\infty}c_{j}\(\frac{z}{k^{b}}\)^{\frac{2j+1}{2(b-1)}}\]
The exponential term ensures convergence in $k$ of the Borel summed transseries. A similar result
can be obtained for $-k.$

Thus, the transseries of $f$ is of the form

\[
\tilde{f}(z)=\Gamma(1+\dfrac{1}{b})z^{-\frac{1}{b}}-\dfrac{1}{2}+\frac{i}{2\pi}\sum_{j=1}^{\infty}(1-(-1)^{jb})\zeta(jb+1)b_{j}z^{j}\]

\[
+\begin{cases}
\frac{i}{2\pi}\sum\limits _{k=1}^{\infty}e^{(b-1)(\frac{2\pi}{b i})^{\frac{b}{b-1}}(\frac{k^{b}}{z})^{\frac{1}{b-1}}}\sum\limits _{j=0}^{\infty}c_{j}k^{\frac{-2jb+2-b}{2(b-1)}}z^{\frac{2j-1}{2(b-1)}} & \,-\frac{\pi}{2}\leqslant\arg z\leqslant\theta_{1}\\
-\frac{i}{2\pi}\sum\limits _{k=1}^{\infty}e^{(b-1)(\frac{2\pi}{b
i})^{\frac{b}{b-1}}(\frac{(-k)^{b}}{z})^{\frac{1}{b-1}}}\sum\limits
_{j=0}^{\infty}c_{j}(-k)^{\frac{-2jb+2-b}{2(b-1)}}z^{\frac{2j-1}{2(b-1)}}
& \,\theta_{2}\leqslant\arg z\leqslant\frac{\pi}{2}\end{cases}\]

\begin{Remark}
 The analysis can be extended to the case

\[
f(z)=\sum_{k=1}^{\infty}F(k)e^{-k^{\bb}z}\:(\lambda\neq0)\]
by noticing that

\[
f(z)=\int_{0}^{\infty}(zF(p)-F'(p))e^{-pz}[p^{\frac{1}{\bb}}]dp\] If
the expression of $F(k)$ is simple, for example $F(k)$ is a finite
combination of terms of the form
$\mu^{k^{\frac{1}{\bb}}}k^{\lambda}(\log k)^{m}\:(m=0,1,2,3...)$,
the method in this section applies with little change to calculate the
transseries of $f(z)$.
\end{Remark}
For special values of $\bb$, asymptotic information as $z$ approaches
the imaginary line can be  obtained in the following way. Let
$z=\delta+2\pi i\beta$ and
\[
f(z)=\sum_{k=1}^{\infty}e^{-k^{\bb}(\delta+2\pi i\beta)}\:(\Re(\delta)>0)\]
If $1<\bb\in\NN$, we
may obtain the asymptotic behavior for all rational $\beta=\dfrac{m}{n}$,
by noting that $e^{-k^{\bb}(2\pi i\beta)}=e^{-(k+n)^{\bb}(2\pi
  i\beta)}$ and splitting the sum into
\[
f(z)=\sum_{j=0}^{\infty}\sum_{l=1}^{n}e^{-(nj+l)^{\bb}\(\delta+2\pi
i\dfrac{m}{n}\)}=\sum_{l=1}^{n}e^{-2\pi
i\dfrac{m}{n}l^{\bb}}\sum_{j=0}^{\infty}e^{-(nj+l)^{\bb}\delta}\]
It follows that
\[
\sum_{j=0}^{\infty}e^{-(nj+l)^{\bb}\delta}=n^{\bb}\delta\int_{l^b/n^b}^{\infty}e^{-n^{\bb}\delta s}\left(s^{\frac{1}{\bb}}-l/n\right)ds\]
\[
=n^{\bb}\delta \int_{l^\bb/n^\bb}^{\infty}e^{-n^{\bb}\delta s}
\(s^{\frac{1}{\bb}}-l/n\)ds-n^{\bb}\delta
\int_{l^\bb/n^\bb}^{\infty}e^{-n^{\bb}\delta s}\left\{s^{\frac{1}{\bb}}-l/n\right\}ds\]
by the argument  above. Since the absolute value of the
fractional part  does not exceed one, we have the
estimate
\[
\sum_{j=0}^{\infty}e^{-(nj+l)^{\bb}\delta }=n^{\bb}\delta \int_{(\frac{l}{n})^{\bb}}^{\infty}e^{-n^{\bb}\delta s}s^{\frac{1}{\bb}}ds+O(1)=\frac{1}{n}{\Gamma\(1+\dfrac{1}{\bb}\)}\delta ^{-\frac{1}{\bb}}+O(1)\:(z\rightarrow0)\]
This implies
\begin{equation}
  \label{eq:deffz}
  f(z)=\frac{1}{n}\sum_{l=1}^{n}e^{-2\pi
i\dfrac{m}{n}l^{\bb}}\Gamma\(1+\dfrac{1}{\bb}\)z^{-\frac{1}{\bb}}+O(1)
\end{equation}
Therefore $f(z)$ either blows up like $z^{-\frac{1}{\bb}}$(when
$\sum_{l=1}^{n}e^{-2\pi i\dfrac{m}{n}l^{\bb}}\neq0$) or it is bounded
(when $\sum_{l=1}^{n}e^{-2\pi i\dfrac{m}{n}l^{\bb}}=0$).

The Fourier expansion of the fractional part can be used to calculate
the transseries as we did for $\beta=0$, but we shall omit the
calculation here.

For special values of $b$, asymptotic information is relatively easy
to obtain on a dense set along the barrier.  This is the case when
$\bb=\dfrac{r+1}{r}$ where $r\in\NN$; then, the transseries contains
exponential sums   in terms of integer powers, $k^{r}$, a
consequence of the duality relation $\dfrac{1}{b}+\dfrac{1}{d}=1$,
which at the transseries level is of the form $\sum
e^{-k^{b}z}\rightarrow g_{1}+\sum e^{ck^{d}z^{-d/b}}g_{2}$ where
$g_{1},g_{2}$ are power series.  We illustrate this for
$\bb=\dfrac{3}{2}$.

Without loss of generality, we assume $\beta<0$. The transseries of
$f$ is given in (\ref{Tr3/2}).  To estimate the asymptotic behavior of
$f(z)$ as $z$ approaches the imaginary line, we rewrite (\ref{Tr3/2}) as
\begin{multline}
  f(z)=\Gamma\(\frac{5}{3}\)z^{-\frac{2}{3}}-\dfrac{1}{2}\\+\frac{i}{2\pi}\sum_{k=1}^{\infty}\(\frac{k}{z}\)^{2}\int_{0}^{\infty}e^{-s\frac{k^{3}}{z^{2}}}\(\dfrac{3}{4\sqrt{2}(\pi
i)^{\frac{3}{2}}}s^{\frac{1}{2}}-\dfrac{3i}{8\pi^{3}}s-\dfrac{105i^{\frac{3}{2}}}{256\sqrt{2}\pi^{\frac{9}{2}}}s^{\frac{3}{2}}+\cdots\)ds\\
+\sum_{k=1}^{\infty}e^{\dfrac{32i\pi^{3}k^{3}}{27z^{2}}}\dfrac{4\sqrt{2}i^{-\frac{1}{2}}\pi}{3}k^{\frac{1}{2}}z^{-1}\\\frac{i}{\pi}\sum_{k=1}^{\infty}\(\frac{k}{z}\)^{2}e^{-\dfrac{32i\pi^{3}k^{3}}{27z^{2}}}\int_{0}^{\infty}e^{-s\frac{k^{3}}{z^{2}}}\(\dfrac{i^{-\frac{1}{2}}}{8\sqrt{2}\pi^{\frac{3}{4}}}(s-s_{0})^{\frac{1}{2}}+\cdots\)ds
\end{multline}
 Watson's Lemma implies that
 \begin{equation}
   \label{eq:eqtest0}
   f(z)=\dfrac{4\sqrt{2}i^{-\frac{1}{2}}\pi}{3z}\sum_{k=1}^{\infty}k^{\frac{1}{2}}e^{\dfrac{32i\pi^{3}k^{3}}{27z^{2}}}+O(1)
 \end{equation}
 The is sum in (\ref{eq:eqtest0}) is  similar to the  sum with
 $\bb=3$, and can be estimated in a similar way:

\[
\sum_{k=1}^{\infty}k^{\frac{1}{2}}e^{-(y+2\pi i\frac{m}{n})}=\(\sum_{l=1}^{n}e^{-2\pi i\dfrac{m}{n}l^{3}}\)\frac{\sqrt{\pi}}{3n\sqrt{y}}+o\(\frac{1}{\sqrt{y}}\)\]

Setting $\dfrac{32i\pi^{3}k^{3}}{27z^{2}}=y+2\pi i\dfrac{m}{n}$ , we
have $z=-\dfrac{4\pi
i}{3\sqrt{3}}\sqrt{\dfrac{n}{m}}$+$\(\dfrac{n}{m}\)^{\frac{3}{2}}y+o(y)$.
The asymptotic behavior can be obtained for $\beta=-\dfrac{4\pi
i}{3\sqrt{3}}\sqrt{\dfrac{n}{m}}$ (this includes all rationals) by
substituting $y=\dfrac{32i\pi^{3}k^{3}}{27z^{2}}-2\pi i\dfrac{m}{n}$
in the above estimates. Setting $z =-\dfrac{4\pi
i}{3\sqrt{3}}\sqrt{\dfrac{n}{m}}$+$\delta$ ($\delta>0$), a direct
calculation shows that
\begin{equation}
  \label{eq:eqtest}
  \sqrt{\Re(z )}f(z )=\frac{\sqrt{6\pi
i}}{3^{\frac{7}{4}}}\(\frac{n}{m}\)^{\frac{1}{4}}\frac{1}{n}\sum_{l=1}^{n}e^{-2\pi
i\dfrac{m}{n}l^{3}}+o(1)
\end{equation}
\subsection{Details of the proof of
 Theorem~\ref{P1}}\label{genblowup}
Consider, more generally,
\[
f(\delta,\beta)=\sum_{k=0}^{\infty}a(k)e^{-g(k)(\delta+2\pi i\beta)}\:, \ \ \delta>0\]
where $g(k)>0$ is a real function
and$\int_{0}^{\infty}|a(t)|^{2}dt=\infty$.

We find  the behavior of
\begin{multline}
\int_{\beta_{0}}^{\beta_{1}}|f(\delta,\beta)|^{2}d\beta\\
=\int_{\beta_{0}}^{\beta_{1}}\sum_{k=0}^{\infty}|a(k)|^{2}e^{-2g(k)\delta}d\beta+\int_{\beta_{0}}^{\beta_{1}}\sum_{k\neq
j}a(k)\bar{a}(j)e^{-(g(k)+g(j))\delta+(g(j)-g(k))2\pi i\beta}d\beta
\\
  =(\beta_{1}-\beta_{0})\sum_{k=0}^{\infty}|a(k)|^{2}e^{-2g(k)\delta}\\
+\frac{1}{2\pi
i}\sum_{k\neq
j}\frac{a(k)\bar{a}(j)}{g(j)-g(k)}e^{-(g(k)+g(j))\delta}e^{(g(j)-g(k))2\pi
i\beta_{0}}\left(e^{(g(j)-g(k))2\pi i(\beta_{1}-\beta_{0})}-1\right)
\end{multline}
where
$\beta_{0,1}\in\RR$ are arbitrary, or after $m$ integrations,
\begin{multline}
  F(\delta)=\int_{\beta_{m-1}}^{\beta_{m-1}+c_{m-1}}\cdots\int_{\beta_{1}}^{\beta_{1}+c_1}\int_{\beta_{0}}^{\beta_{0}+c_0}|f(\delta,\beta)|^{2}d\beta d\beta_{0}\cdots
  d\beta_{m-2}\\
  =c_0c_1\cdots c_m\sum_{k=0}^{\infty}|a(k)|^{2}e^{-2g(k)\delta}+O\left(\sum_{k\neq j}\left|\frac{a(k)\bar{a}(j)}{(g(j)-g(k))^{m}}\right|e^{-(g(k)+g(j))\delta}\right)
\end{multline}

Note that
\begin{multline}
 \sum_{k\neq
j}\left|\frac{a(k)\bar{a}(j)}{(g(j)-g(k))^{m}}\right|e^{-(g(k)+g(j))\delta}\\=2\sum_{k=0}^{\infty}\sum_{n=1}^{\infty}\left|\frac{a(k)\bar{a}(k+n)}{(g(k+n)-g(k))^{m}}\right|e^{-(g(k)+g(k+n))\delta}\\
=2\sum_{k=0}^{\infty}|a(k)|e^{-2g(k)\delta}\sum_{n=1}^{\infty}\left|\frac{a(k+n)}{(g(k+n)-g(k))^{m}}\right|e^{-(g(k+n)-g(k))\delta}\\
=O\left(\sum_{k=0}^{\infty}|a(k)|e^{-2g(k)\delta}\sum_{n=1}^{\infty}\left|\frac{a(k+n)}{(g(k+n)-g(k))^{m}}\right|\right)
\end{multline}
under our assumption.

If furthermore we have
\[
\sum_{n=1}^{\infty}\left|\frac{a(k+n)}{(g(k+n)-g(k))^{m}}\right|=o(a(k))\]
then we obtain
\begin{equation}
  F(\delta)=c_0c_1\cdots c_m\sum_{k=0}^{\infty}|a(k)|^{2}e^{-2g(k)\delta}+O\left(\sum_{k=0}^{\infty}\frac{|a(k)|^{2}}{G(k)}e^{-2g(k)\delta}\right)
\end{equation}
where $G(k)>1,\, G(k)\rightarrow\infty$ as $k\rightarrow\infty$,
or
\begin{equation}
  \label{eq:eqB}
 F(\delta)\left(\sum_{k=0}^{\infty}|a(k)|^{2}e^{-2g(k)\delta}\right)^{-1} =c_0\cdots c_m+o(1)\ \ as\ \delta\to 0^+
\end{equation}
The result now follows from the following lemma.
\begin{Proposition}
  Assume

(i) $h_n:\RR\to [0,\infty)$ are locally $L^1$, and

(ii) $\displaystyle \lim_{n\to\infty}\int_{B} h_n(x_1+\cdots +x_N)dx_1\cdots dx_N= meas(B)$ for any box $B\displaystyle=\prod_{i=1}^N[a_i,b_i]$.

Then
 $h_n\to 1$ in the dual of $C[\alpha,\beta]$ for any $[\alpha,\beta]$.

\end{Proposition}

\begin{proof}
  We first take $N=2$, the general case will follow by induction on $N$.

  Consider the rectangle $B_{c}=(a,b-c)\times (0,c)$, $0<c<b-a$. By changing coordinates
to $x+y=s,y=y'$, we get that
\begin{equation}
  \label{eq:Tc}
 c^{-1}\int_{B} h_ndydx= \int_a^bh_n(s)T_{a,b;c}(s)ds\to meas(T_{a,b;c});\ \text{as } n\to\infty
\end{equation}
where $T_{a,b;c}(\cdot)$ is the function having $T_{a,b;c}$ as a
graph, $T_{a,b;c}$ being an isosceles trapezoid with lower base the
interval $(a,b)$ and upper base of length $b-a-2c$ at height 1.

We also note that the indicator function of $[a,b]$, $\mathbf{1}_{ab}$,
satisfies the inequalities $T_{a-c,b+c;c}\ge \mathbf{1}_{ab}\ge
T_{a,b;c}$. Thus, since $c$ is arbitrary and $h_n\ge 0$, and both
$meas (T_{a-c,b+c;c})$, and $meas (T_{a,b;c})$ tend to $(b-a)$ as
$c\to 0$, we have
\begin{equation}
  \label{eq:(b-a)}
  \lim_{n\to\infty}\int_a^bh_n(s)ds =(b-a) =meas([a,b])
\end{equation}
In particular, given $\alpha<\beta$, $\|h_n\|_{L^1[\alpha,\beta]}$ are uniformly
bounded, that is, for some $C\ge 1$ we have
\begin{equation}
  \label{eq:L1}
 \sup_{n\ge 1}\|h_n\|_{L^1[\alpha,\beta]}\le C(\beta-\alpha)
\end{equation}
Since a continuous function on $[\alpha,\beta]$ is approximated
arbitrarily well in sup norm by finite linear combinations of
indicator functions of intervals, it follows from (\ref{eq:(b-a)}),
(\ref{eq:L1}) \footnote{Alternatively, and somewhat more compactly,
one can prove the result without the intermediate steps (\ref{eq:(b-a)})
and (\ref{eq:L1}) by upper and lower bounding continuous functions
by sums of trapezoids.}
 and the triangle inequality that
\begin{equation}
  \label{eq:concl}
  \int_\alpha^\beta h_n(s)f(s)ds\to\int_\alpha^\beta f(s)ds, \ \forall\,f\in C[\alpha,\beta]
\end{equation}
For general $N$ we use $h^-_n(s)=\int_{B'}h_n(s+x_1+...+x_{N-1})dx_1\cdots dx_{N-1}$  and
(\ref{eq:(b-a)}) to reduce the problem to $N-1$.
\end{proof}

The condition

\[
\sum_{n=1}^{\infty}\left|\frac{a(k+n)}{(g(k+n)-g(k))^{m}}\right|=o(a(k))\]
is satisfied, for instance, if

(1) $\exists c>0$ so that $c<|a(k)|<c^{-1}$, or $|a(k)|$
decreases to 0. (Note that $g(k+n)-g(k)=g'(k+tn)n$, where
$0\leqslant t\leqslant1$, and $g'(k)\rightarrow\infty$ as
$k\rightarrow\infty$.)

(2) $\exists c,r$ so that $c<a(k)<k^{r}$. $g'(k)\geqslant
k^{\varepsilon}$ for some small $\varepsilon>0$.
\section{Proof of Theorem~\ref{T1}}\label{S5}
In Appendix \S\ref{Fi} we list some known facts about iterations of  maps.

\begin{proof}[Proof of B\"otcher's theorem, for (\ref{eq:eqG})]
  ({\bf Note:} this line of proof extends to general analytic maps.)

 We
write $\psi=\lambda z+\lambda^2zg(z)$ and obtain
\begin{equation}
  \label{eq:H}
  g(z)-\frac{1}{2}g(z^2)=\frac{1}{2}z+\frac{1}{2}\lambda\left[g(z)(z-g(z))+g(z^2)\right]+\frac{\lambda^2 z}{2}g(z)g(z^2)=N(g)
\end{equation}
Let $\mathcal{A_{\lambda}}$ denote the functions analytic
in the polydisk $\mathbb{P}_{1,\epsilon}=\mathbb{D}\times
\{\lambda:|\lambda|<\epsilon\}$. We write (\ref{eq:H}) in the form (see \ref{eq:defT}))
\begin{equation}
  \label{eq:eqH4}
  g=2\mathfrak TN(g)
\end{equation}
This equation is manifestly contractive in the sup norm, in a ball of
radius slightly larger than $1/2$ in $\mathcal{A_{\lambda}}$, if
$\epsilon$ is small enough.  For $\lambda\ne 0$, evidently
$\psi=\phi^{-1}$ is also analytic at zero. \end{proof}
\begin{Lemma}
  $\psi$ is analytic in
$\mathbb{D}_1$ for all $\lambda$ with $|\lambda|<1$.
\end{Lemma}
\begin{proof} We have
\begin{equation}
  \label{eq:eqGM}
  \psi(z)=\frac{\lambda}{2}\left( X+\sqrt{X^2+4X/\lambda}\right)=:F(X);\ X=\psi(z^2)
\end{equation}
For small $z\ne 0$, $\psi(z)=O(z)$ and thus $F(\psi(z^2))$ is well defined and analytic. Note that (\ref{eq:eqGM}) provides analytic
continuation
of $\psi$ from $\mathbb{D}_{\rho^2}$ to  $\mathbb{D}_{{\rho}}$, provided nowhere
in $\mathbb{D}_{\rho^2}$ do we have $\psi=-4/\lambda$  (certainly the case if $\rho$ is small). We assume, to get a contradiction, that there is a $z_0$,
$|z_0|=\lambda_0<1$  so that $\psi(z_0)=-4/\lambda$, and we choose the least
$\lambda_0$ with this property. By the previous discussion, $\psi$
is analytic in the open disk $\mathbb{D}_{\sqrt{\lambda_0}}$.
 Then we use the ``backward'' iteration  $\psi(z^2)=
\lambda^{-1}\psi(z)^2/(1+\psi(z))$ to calculate $\psi(z^2)$
from $\psi(z)$,  starting with $z=z_0$. This is in fact
equivalent to  (\ref{eq:infty1}); after the substitution $x=(-\lambda y)^{-1}$
we return to (\ref{eq:logist1}), with $x_0=\lambda/4$. Using (vi) and
(vii) of \S\ref{Fi}, it follows that $1/x_n\not\to 0$, that is,  $\psi(z_0^{2^n})\not \to 0$. This impossible, since $\psi$ is analytic and $\psi(0)=0$.
\end{proof}

\begin{proof}[Proof of Theorem~\ref{T1}, (i)]
We return to
(\ref{eq:eqG}). Taking $a\in (0,1)$ $m_n=\sup \{|\psi(z)|:
|z|<a^{1/2^n}$ we note that
\begin{equation}
  \label{eq:eqm}
  m_{n+1}\le \frac{1}{2} |\lambda|(m_n+\sqrt{m_n^2+4 m_n/|\lambda|})
\end{equation}
The sequence of $m_n$ is bounded by the sequence of  $M_n$, defined  by replacing
``$\le$'' with ``$=$'' in (\ref{eq:eqm}).  Since $\frac{1}{2}
|\lambda|(x+\sqrt{x^2+4 x/|\lambda|})<x$ if
$x>A:=|\lambda|/(1-|\lambda|)$, we have $\limsup_n M_n\le A$. By the
maximum principle, $|\psi(\lambda,z)|< A$ in
$\mathbb{D}\times\mathbb{D}$. Thus, by Cauchy's formula in $\lambda$
we have $|\psi_n(z)|\le A$ for all $n$ and $z\in \mathbb{D}$. The
radius of convergence of (\ref{eq:formG}) in $\lambda$ is at least
one.  By \S\ref{Fi},  (vii), the radius of convergence is exactly
one. 

Indeed, note first that  (a)  if $\psi$ is analytic in
$\mathbb{D}$ then $\psi'\ne 0 $ in $\mathbb{D}$, otherwise 
$\psi'(z_1)=0$ would imply  $\psi'(z_1^{2^n})=0$ in contradiction
with $\psi'(0)=\lambda$. This means that if there is a $z_0$, $\psi(z_0)=-4/\lambda$, then $\sqrt{z_0}$ is a singular point of $\psi$. 

Secondly, any $\lambda$ of the form $1+i\epsilon$ with small
$\epsilon$ correspond to $c=1/4+1/4 \epsilon^2,$ outside the
Mandelbrot set.  Thus, in the iteration (\ref{eq:infty1}), the initial
condition $y_0=-4/\lambda$ implies $y_n\to 0$. We can now use the
implicit function theorem to suitably match $y_n$, once it is small
enough, to some value of $\psi$ near zero.  Indeed, the equation
$y_n=\lambda z_{0}^{2^n}+O(z_0^{2^{n+1}})$ has $2^n$ solutions. This
means that for such a $z_0$, using (\ref{eq:eqGM}) to iterate
backwards and to determine $\psi(z_0)$ (noting the parallel to
(\ref{eq:infty1})), we have $\psi(z_0)=-4/\lambda$, and by (a) above,
$\psi$ cannot be analytic in $z$ in $\mathbb{D}$.

Formula (\ref{eq:frakt}) follows by straightforward expansion of (\ref{eq:eqG})
and identification of powers of $\lambda$.\end{proof}
\begin{proof}[Proof of Theorem~\ref{T1}, (ii)]
The stated type of lacunarity  of $\psi_k$ follows from (\ref{eq:frakt}) by induction,
noting the discrete convolution structure in $k$.\end{proof}
\begin{proof}[Proof of Theorem~\ref{T1}, (iii)]
Continuity of $\psi_k$ in $\overline{\mathbb{D}}$ also follows by induction from (\ref{eq:frakt}) and the
properties of $\mathfrak{T}$. By  dominated convergence (applied to the
discrete measure $|\lambda|^n$),  for $\lambda<1$, $\psi$ is
continuous in $\overline{\mathbb{D}}$ and the Fourier series converges
pointwise in $\partial{\mathbb{D}}$.

To show convergence of the Fourier series of $H$ we only need to show
$\inf_{\mathbb{D}}|\psi|>0$. Now, $\psi$ clearly cannot
vanish for any $z_0=\mathbb{D}$, otherwise $\psi(z_0^{2^n})=0$, would imply
by analyticity $\psi\equiv 0$. If
$\min_{\mathbb D_{\rho}}|\psi(\rho)|=\epsilon$ would be small enough, then
$\min_{\mathbb D_{\rho^2}}|\psi(\rho)|\le O(\epsilon^2)\ll
\epsilon$, contradicting the maximum principle for $z/\psi(z)$.

The rest of the proof is straightforward
calculation, using the analyticity of $\psi$. \end{proof} The extension of the small
$\lambda$ analysis  to higher
order polynomials is also straightforward.
\begin{Note}\label{Explc} The transseries of the B\"otcher map at binary rational
numbers can be calculated rather explicitly. This is beyond
the scope here, and will be the subject of a different paper. A less explicit
expression has been obtained in \cite{CPAM}. We note that the constant
    $\log_2(2\pi)$ in (1.17) of \cite{CPAM} should be $(2\pi)/\log
    2$.
  \end{Note}
\section{Appendix}\label{appendix}
\subsection{Proof of Lemma~\ref{linv}}
 For every $m$, we write $g(s)=g_m+o(s^{-m})$ where $g_m$
is a finite sum, an initial sum in the asymptotic series
of $g$. It is straightforward to show that $g_m^{-1}(y)$ has
an asymptotic power series as $y\to\infty$. Then, in the equation
$g(s)=y$ we write $s=s_m+\epsilon$ where $g_m(s_m)=y$. Then,
$y=g(s)=g(s_m+\epsilon)=g_m(s_m)+g'(\xi)\epsilon+(g(s_m)-g_m(s_m))$
implies $\epsilon=(g_m(s_m)-g(s_m))/g'(\xi)=o(s^{-m-\theta})$ where  $g'(\xi)\sim a x^{\theta}$. Now $\theta$ is fixed and $m$ is arbitrary, and then the result
follows.
\subsection{Proof of Lemma~\ref{eq:asympt1}}
  We have
  \begin{multline}
  \label{eq:34}
f-e^{-xg(0)}=-\sum_{k=1}^{\infty}k(e^{-xg(k+1)}-e^{-xg(k)})=x\int_0^{\infty} e^{-xg(s)}g'(s)\lfloor s \rfloor ds\\=x\int_{g(0)}^{\infty} e^{-xu}\lfloor g^{-1}(u)\rfloor du=x\int_{g(0)}^{\infty} e^{-xu}  g^{-1}(u)du-
x\int_{g(0)}^{\infty} e^{-xu} \{g^{-1}(u)\}  du
\end{multline}
and we also have
\begin{equation}
  \label{eq:ginv}
  x\int_{g(0)}^{\infty} e^{-xu}  g^{-1}(u)du=-\int_{0}^{\infty} (e^{-xg(u)})'u du=
\int_{0}^{\infty} e^{-xg(u)}du
\end{equation}
whereas
$$0<\int_0^{\infty} e^{-xu} xg'(u)\{u\}du\le \int_0^{\infty} e^{-xu} xg'(u)du=e^{xg(0)}$$
\subsection{Proof of Lemma \ref{Lgeom}}\label{pfLgeom}
By the Fourier coefficients formula we have
\begin{multline}
c_{0}=\int_{0}^{1}f(a^{y})dy+\int_{0}^{1}ydy-\int_{0}^{1}\check{G}(a^{y})dy\\
c_{k}=\int_{0}^{1}f(a^{y})e^{-2k\pi iy}dy+\int_{0}^{1}ye^{-2k\pi iy}dy-\int_{0}^{1}\check{G}(a^{y})e^{-2k\pi iy}dy\\
=\frac{i}{2k\pi}+\sum_{n=0}^{\infty}\int_{0}^{1}e^{-a^{n+y}-2k\pi iy}dy+\sum_{n=1}^{\infty}\int_{0}^{1}\frac{(-1)^{n}a^{ny}}{n!(a^{n}-1)}e^{-2k\pi iy}dy\\
 =\frac{i}{2k\pi}+\frac{1}{\log a}\sum_{n=0}^{\infty}a^{\frac{2kn\pi i}{\log a}}\(\Gamma\(-\frac{2k\pi i}{\log a},a^{n}\)-\Gamma\(-\frac{2k\pi i}{\log a},a^{1+n}\)\)\\+\sum_{n=1}^{\infty}\frac{(-1)^{n+1}}{n!(2k\pi i-n\log a)} =\sum_{n=0}^{\infty}\frac{(-1)^{n+1}}{n!(2k\pi i-n\log a)}+\frac{1}{\log a}\Gamma\(-\frac{2k\pi i}{\log a},1\)\:(k\neq0)
\end{multline}

Note that since $\mathcal{L}^{-1}(\frac{1}{2k\pi i-n\log a})=\dfrac{1}{\log a}e^{\frac{2k\pi ip}{\log a}}\:(n\rightarrow p)$
we have
\begin{multline*}
\sum_{n=0}^{\infty}\frac{(-1)^{n+1}}{n!(2k\pi i-n\log
a)}=\int_{0}^{\infty}\sum_{n=0}^{\infty}\frac{(-1)^{n+1}e^{-np}}{n!}\dfrac{1}{\log
a}e^{\frac{2k\pi ip}{\log a}}dp\\=\dfrac{1}{\log
a}\int_{0}^{\infty}e^{-e^{-p}}e^{\frac{2k\pi ip}{\log
a}}dp\\=\dfrac{1}{\log a}\int_{0}^{1}e^{-t}t^{\frac{-2k\pi i}{\log
a}-1}dt=\frac{1}{\log a}\(\Gamma(-\frac{2k\pi i}{\log
a})-\Gamma(-\frac{2k\pi i}{\log a},1)\)
\end{multline*}

The above procedure is justified for $k$ in the upper half plane.
By analytic continuation the expression holds for $k$ real as well.
Eq. (\ref{eq:ck}) follows.

We can further resum the series in the above expression by noting
that
\begin{multline}
   \sum_{k\neq0}\Gamma\(-\frac{2k\pi i}{\log a}\)
z^{\frac{2k\pi i}{\log a}}=\sum_{k=1}^{\infty}
\frac{-\log a}{2k\pi i}\int_{0}^{\infty}
t^{\frac{-2k\pi i}{\log a}}e^{-t}z^{\frac{2k\pi i}{\log a}}dt
\\+\sum_{k=-\infty}^{-1}\frac{-\log a}{2k\pi i}=\int_{0}^{\infty}\sum_{k=1}^{\infty}\frac{-\log a}{2k\pi i}
\(\frac{z}{t}\)^{\frac{2k\pi i}{\log a}}e^{-t}dt
\int_{0}^{\infty}t^{\frac{-2k\pi i}{\log a}}e^{-t}z^{\frac{2k\pi i}{\log a}}dt
\\+
\int_{0}^{\infty}\sum_{k=-\infty}^{-1}\frac{-\log a}{2k\pi i}
\(\frac{z}{t}\)^{\frac{2k\pi i}{\log a}}e^{-t}dt=\frac{\log a}{2\pi i}\int_{0}^{\infty}
\log_{\RR}\(1-\(\frac{z}{t}\)^{\frac{2\pi i}{\log a}}\)e^{-t}dt
\\-
\frac{\log a}{2\pi i}\int_{0}^{\infty}
\log_{\RR}\(1-\(\frac{z}{t}\)^{-\frac{2\pi i}{\log a}}\)e^{-t}dt=\frac{\log a}{2\pi i}\int_{0}^{\infty}
\log_{\RR}\(-\(\frac{z}{t}\)^{\frac{2\pi i}{\log a}}\)e^{-t}dt
\\
=-\frac{\log a z}{2\pi i}\int_{0}^{\infty}\log_{\RR}\(-s^{\frac{2\pi
i}{\log a}}\)e^{-sz}ds
\end{multline}
which can be justified by analytic continuation, for the last
expression and the sum are both analytic, and equal to each other on
the real line. The logarithm $\log_{\RR}$ is defined with a branch cut along $\RR^-$.

We finally obtain the integral representation (\ref{eq:resum}),
valid for $z$ in the right half plane.

\begin{Remark}
Since $\log_{\RR}\(-s^{\frac{2\pi i}{\log a}}\)=i\arg\(\dfrac{2\pi \log
s}{\log a}-\pi\)=2\pi i \(\left\{\dfrac{\log s}{\log a}\right\}-\dfrac{1}{2}\)$,
we actually recover the last term of (\ref{eq:asympt1}).
\end{Remark}

\subsection{Direct calculations for $b\in\NN$ integer; the cases $b=3,b=3/2$}\label{Direct,3}
\begin{Proposition}
If $\bb$ is an integer, the behavior of \[
f(\delta+2\pi i \beta)=\sum_{k=1}^{\infty}e^{-k^{\bb}(\delta+2\pi i\beta)}\:(\Re(\delta)>0)\]
where $\beta=2\pi i m/n$, m and n being integers, as $\delta$ approaches
0 is
\[
f(\delta+2\pi i \beta)=\left[\frac{1}{n}\sum_{l=1}^{n}e^{-2\pi
i\dfrac{m}{n}l^{\bb}}\Gamma\(1+\dfrac{1}{\bb}\)\right]\delta^{-\frac{1}{\bb}}+O(1)\]
Therefore $f(\delta)$ either blows up like $\delta^{-\frac{1}{\bb}}$ or is
bounded.

The more general case $\bb=\dfrac{r+1}{r}$ where $r$ is an integer
can be treated similarly. In particular, if
$\bb=\dfrac{3}{2}$, for $\beta=-\dfrac{4\pi
  i}{3\sqrt{3}}\sqrt{\dfrac{n}{m}}$(this includes all rational
numbers), we have
 \[
\sum_{k=1}^{\infty}e^{-k^{3/2}z}=\frac{\sqrt{6\pi
i}}{3^{\frac{7}{4}}m^{\frac{1}{4}}n^{\frac{3}{4}}}\sum_{l=1}^{n}e^{-2\pi
i\dfrac{m}{n}l^{3}}\frac{1}{\sqrt{\delta}}+o\left(\frac{1}{\sqrt{\delta}}\right)\]
with $z=-\dfrac{4\pi
i}{3\sqrt{3}}\sqrt{\dfrac{n}{m}}$+$\delta$($\delta\to 0^+$)
\end{Proposition}

\begin{proof}

In general, to find the asymptotic behavior of $f(z)$ we analyze
the functions
\begin{equation}
  \label{eq:deffk}
 f_{k}(z)=\int_{0}^{\infty}e^{-pz-2k\pi
ip^{\frac{1}{\bb}}}dp,\ k\in\ZZ
\end{equation}
Letting $p=q\(\dfrac{z}{k}\)^{\frac{\bb}{1-\bb}}$ we have

\[
\int_{0}^{\infty}e^{-pz-2k\pi
ip^{\frac{1}{\bb}}}dp=\(\frac{k}{z}\)^{\frac{\bb}{\bb-1}}\int_{0}^{\infty}e^{-(q+2\pi
iq^{\frac{1}{\bb}})(\frac{k^{\bb}}{z})^{\frac{1}{\bb-1}}}dq\]

Next we let $s=h(q)=q+2\pi iq^{\frac{1}{\bb}}$and
$f_{k}(z)=(\frac{k}{z})^{\frac{\bb}{\bb-1}}\intop_{C_1}e^{-s(\frac{k^{\bb}}{z})^{\frac{1}{\bb-1}}}\ds\frac{1}{h'(h^{-1}(s))}ds$
where the contour $C_1$ is a curve from the origin to $\infty$ in
the first quadrant.

We can find the asymptotic behavior of $f_{k}(z)$ using Watson's
Lemma \cite{benderorszag}. By iterating the contractive map
$q\rightarrow\(\dfrac{s-q}{2\pi i}\)^{\bb}$ near 0, we can easily
see that $\dfrac{h^{-1}(s)}{s^{\bb}}$ is analytic in $s^{\bb-1}$,
which implies $\ds\frac{1}{h'(h^{-1}(s))}$ is analytic in
$s^{\bb-1}$ near 0 with no constant term.

Now let's consider the examples $\bb=3$ and $b=3/2$; for  $\bb=3$ we
have $s=h(q)=q+2\pi iq^{\frac{1}{3}}$ and
$\dfrac{ds}{dq}=$$\dfrac{3i}{8\pi^{3}}s^{2}+\dfrac{15}{64\pi^{6}}s^{4}-\dfrac{21i}{128\pi^{9}}s^{6}\cdots$. Thus, the asymptotic power series is

\[
f(z)\sim\Gamma\(\frac{4}{3}\)z^{-\frac{1}{3}}-\dfrac{1}{2}-\dfrac{z}{120}+\frac{z^{3}}{792}+\cdots\]

The branch point of $h^{-1}$  is located at
$s_{0}=\pi^{\frac{3}{2}}\dfrac{4\sqrt{2}}{3\sqrt{3}}(-1)^{\frac{1}{4}}$,
which is between the contour $C_1$ defined above and the $x$-axis. As we start
rotating $z$ from $z>0$, we have, cf (\ref{eq:deffk}),
\begin{multline}
  f_{k}(z)=\(\frac{k}{z}\)^{\frac{3}{2}}\intop_{C_1}\frac{e^{-s(\frac{k^{3}}{z})^{\frac{1}{2}}}}{h'(h^{-1}(s))}ds\\
=\(\frac{k}{z}\)^{\frac{3}{2}}\intop_{0}^{\infty}\frac{e^{-s(\frac{k^{3}}{z})^{\frac{1}{2}}}}{h'(h^{-1}(s))}ds+\(\frac{k}{z}\)^{\frac{3}{2}}e^{-s_{0}(\frac{k^{3}}{z})^{\frac{1}{2}}}\intop_{C_2}\frac{e^{-s(\frac{k^{3}}{z})^{\frac{1}{2}}}}{h'(h^{-1}(s+s_{0}))}ds\\
=\(\frac{k}{z}\)^{\frac{3}{2}}\int_{0}^{\infty}\frac{e^{-s(\frac{k^{3}}{z})^{\frac{1}{2}}}}{h'(h^{-1}(s))}ds+2\(\frac{k}{z}\)^{\frac{3}{2}}e^{-s_{0}(\frac{k^{3}}{z})^{\frac{1}{2}}}\int_{0}^{\infty}\frac{e^{-s(\frac{k^{3}}{z})^{\frac{1}{2}}}}{h'(h^{-1}(s+s_{0}))}ds
\end{multline}
where the contour $C_2$ starts at $\infty$, goes clockwise around
the origin, then ends at $\infty.$ Since now

$$\dfrac{ds}{dq}=\dfrac{(-\pi
i)^{\frac{3}{4}}}{6^{\frac{1}{4}}}(s-s_{0})^{-\frac{1}{2}}+\dfrac{5}{6}+\dfrac{5}{16}\(\dfrac{i}{6\pi}\)^{\frac{3}{4}}(s-s_{0})^{\frac{1}{2}}+\cdots$$

We have
\begin{multline}
  \(\frac{k}{z}\)^{\frac{3}{2}}e^{-s_{0}z}\intop_{C_2}\frac{e^{-s(\frac{k^{3}}{z})^{\frac{1}{2}}}}{h'(h^{-1}(s+s_{0}))}ds\\=e^{-\pi^{\frac{3}{2}}\frac{4\sqrt{2}}{3\sqrt{3}}(-1)^{\frac{1}{4}}z}\(\(\dfrac{\pi
i}{6}\)^{\frac{1}{4}}k^{-\frac{1}{4}}z^{-\frac{1}{4}}+\dfrac{5}{32}\dfrac{i^{\frac{7}{4}}}{6^{\frac{3}{4}}\pi^{\frac{5}{4}}}k^{-\frac{7}{4}}z^{\frac{1}{4}}+\cdots\)
\end{multline}

Therefore the transseries is
\begin{multline}
  \label{eq:trans1}
  \tilde{f}(z)=\Gamma(\frac{4}{3})z^{-\frac{1}{3}}-\dfrac{1}{2}-\dfrac{z}{120}+\frac{z^{3}}{792}+\cdots\\
+\sum_{k=1}^{\infty}e^{-\pi^{\frac{3}{2}}\frac{4\sqrt{2}}{3\sqrt{3}}(-1)^{\frac{1}{4}}(\frac{k^{3}}{z})^{\frac{1}{2}}}\left[\(\dfrac{\pi
i}{6}\)^{\frac{1}{4}}k^{-\frac{1}{4}}z^{-\frac{1}{4}}+\dfrac{5}{32}\dfrac{i^{\frac{7}{4}}}{6^{\frac{3}{4}}\pi^{\frac{5}{4}}}k^{-\frac{7}{4}}z^{\frac{1}{4}}+\cdots\right]\\
-\sum_{k=1}^{\infty}e^{-\pi^{\frac{3}{2}}\frac{4\sqrt{2}}{3\sqrt{3}}(-1)^{\frac{1}{4}}(\frac{-k^{3}}{z})^{\frac{1}{2}}}\left[\(\dfrac{\pi
i}{6}\)^{\frac{1}{4}}(-k)^{-\frac{1}{4}}z^{-\frac{1}{4}}+\dfrac{5}{32}\dfrac{i^{\frac{7}{4}}}{6^{\frac{3}{4}}\pi^{\frac{5}{4}}}(-k)^{\frac{7}{4}}z^{\frac{1}{4}}+\cdots\right]
\end{multline}
The calculation for $\bb=\dfrac{3}{2}$ is similar: in this case
$\dfrac{ds}{dq}=\dfrac{3}{4\sqrt{2}(\pi
i)^{\frac{3}{2}}}s^{\frac{1}{2}}-\dfrac{3i}{8\pi^{3}}s-\dfrac{105i^{\frac{3}{2}}}{256\sqrt{2}\pi^{\frac{9}{2}}}s^{\frac{3}{2}}+\cdots$
and the asymptotic power series is

\[
f(z)\sim\Gamma\(\frac{5}{3}\)z^{-\frac{2}{3}}-\dfrac{1}{2}-\dfrac{3\zeta(\frac{5}{2})}{16\pi^{2}}z+\frac{1}{240}z^{2}+\frac{315\zeta(\frac{11}{2})}{2048\pi^{5}}z^{3}+\cdots\]

The exponential sum is slightly different than in the previous case, for
now the branch point $s_{0}=-\dfrac{32}{27}i\pi^{3}$ lies in the
lower half plane, which means the contour $C_2$ can be deformed
to $[0,$$+\infty)$ without passing through any singularity.

We collect the contribution from the branch point only when $\arg z$
decreases to $-\dfrac{\pi}{4}$from 0. Since
$$\dfrac{ds}{dq}=\dfrac{-4\sqrt{2}i^{\frac{1}{2}}\pi^{\frac{3}{2}}}{3}(s-s_{0})^{-\frac{1}{2}}+\dfrac{4}{3}+\dfrac{i^{-\frac{1}{2}}}{8\sqrt{2}\pi^{\frac{3}{4}}}(s-s_{0})^{\frac{1}{2}}+\cdots$$
we have for the exponential part of the sum
$$\(\dfrac{4\sqrt{2}i^{-\frac{1}{2}}\pi}{3}k^{\frac{1}{2}}z^{-1}+\dfrac{i^{\frac{1}{2}}}{16\sqrt{2}\pi^{\frac{5}{4}}}k^{-\frac{5}{2}}z^{2}+\cdots\)\exp\({\dfrac{32i\pi^{3}k^{3}}{27z^{2}}}\)$$
Therefore, for small $z$ in the right half plane, the transseries  is given by
\begin{multline}
  \label{Tr3/2}
\tilde{f}(z)=\Gamma\(\frac{5}{3}\)z^{-\frac{2}{3}}-\dfrac{1}{2}-\dfrac{3\zeta(\frac{5}{2})}{16\pi^{2}}z+\frac{1}{240}z^{2}+\frac{315\zeta(\frac{11}{2})}{2048\pi^{5}}z^{3}+\cdots\\+\begin{cases}
\sum\limits _{k=1}^{\infty}e^{\frac{32i\pi^{3}k^{3}}{27z^{2}}}\(\frac{4\sqrt{2}i^{-\frac{1}{2}}\pi}{3}k^{\frac{1}{2}}z^{-1}+\frac{i^{\frac{1}{2}}}{16\sqrt{2}\pi^{\frac{5}{4}}}k^{-\frac{5}{2}}z^{2}+\cdots\); \,\,-\frac{\pi}{2}\leqslant\arg z\leqslant-\frac{\pi}{4}\\
-\sum\limits
_{k=1}^{\infty}e^{-\frac{32i\pi^{3}k^{3}}{27z^{2}}}\(\frac{4\sqrt{2}i^{-\frac{1}{2}}\pi}{3}(-k)^{\frac{1}{2}}z^{-1}+\frac{i^{\frac{1}{2}}}{16\sqrt{2}\pi^{\frac{5}{4}}}(-k)^{-\frac{5}{2}}z^{2}+\cdots\);\,\frac{\pi}{4}\leqslant\arg
z\leqslant\frac{\pi}{2}\end{cases}
\end{multline}
The effect of the exponential part of the transseries affects the leading
order when
$z\rightarrow0$ nearly tangentially to the imaginary line.

For example, to see the effect of the first exponential term for
$\bb=\dfrac{3}{2}$, we let $z^{-\frac{2}{3}}\sim
e^{\frac{32i\pi^{3}}{27z^{2}}}z^{-1}$ and $z=-ire^{i\theta}$ near
the negative imaginary line.

The critical curve along which the power term and the exponential
term are of equal order is
$\theta\sim\dfrac{9}{64\pi^{3}}r^{2}\log\dfrac{1}{r}$.
\end{proof}
\begin{figure}[h!]
   $ $ \hskip -3cm
  \includegraphics[scale=0.38]{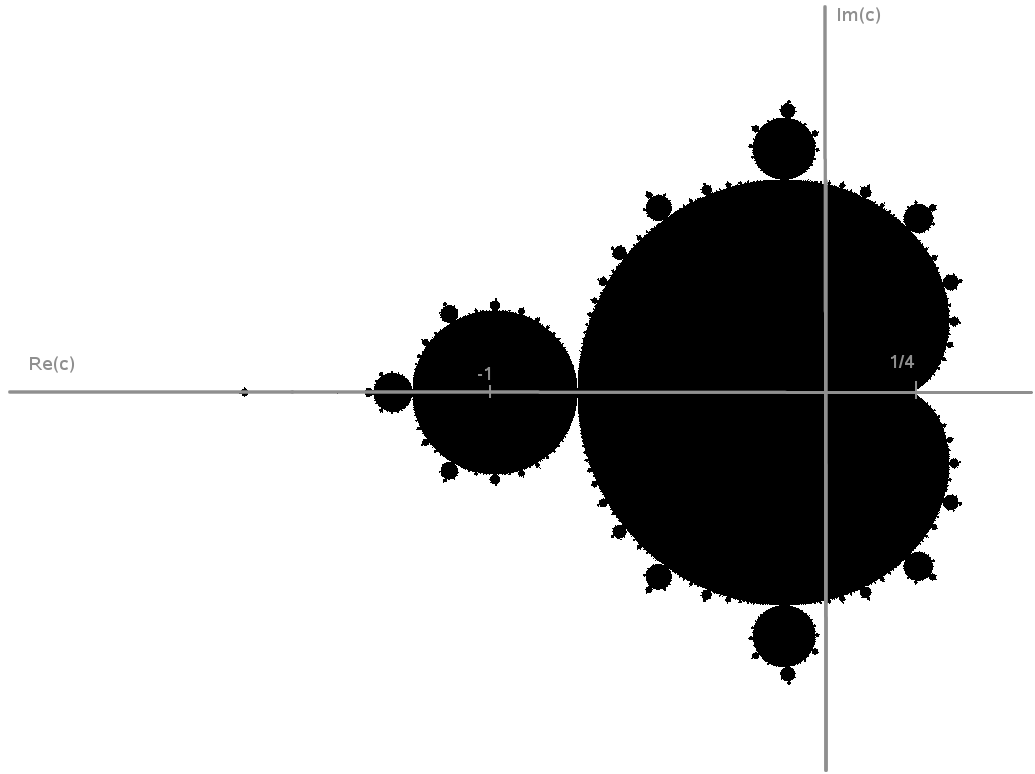}
  \caption{The Madelbrot set (drawn with xaos 3.1 \cite{xaos}).} \label{M1}
  \label{fig:12}
\end{figure}
\subsection{Notes about iterations of maps}\label{Fi}

For the following, see e.g., \cite{Beardon,Devaney,Milnor}.
    \begin{enumerate}
   \item   For $|\lambda|<1$, three types of behavior are possible for
the solution of (\ref{eq:logist1}): if the initial condition $x_0\in \mathcal{F}$, the connected component
of the origin in {\em the Fatou set}, then $x_n\to 0$ as $n\to \infty$.
Clearly,  $x_0\in \mathcal{F}$ if $x_0$ is small enough. If $x_0\in \overline{\mathcal{F}}^c$, the connected
component of infinity in the Fatou set, then
$|x_n|\to\infty$. Clearly, $x_0\in \overline{\mathcal{F}}^c$ if it is large enough. Finally, for   $x_0\in \partial{F}=J$, the Julia set, a (connected) curve
of  nontrivial Hausdorff dimension invariant under
the map, $\{x_n\}_n$ are  dense in $J$ and the evolution is chaotic.

\item $J$ is the closure of the set of repelling periodic points.

\item For polynomial maps, and more generally, for entire maps, $J$ is the boundary of the set of points which converge to infinity under iteration.

\item If the maximal disk of analyticity of $\psi$ is the unit disk $\mathbb{D}_1$, then $\psi$ maps $\mathbb{D}_1$ biholomorphically onto the immediate basin $\mathcal{A}_0$ of zero.  If on the contrary the maximal disk is $\mathbb{D}_r,\,r<1$, then there is at least one other critical point in $\mathcal{A}_0$, lying in $\psi(\partial\mathbb{D}_r)=J_y$, the Julia set of \ref{eq:infty1}.

\item If $r=1$, it follows that $\psi(\partial\mathbb{D}_1)=J_y$.

\item By the change
of variable $x_n=-(1/\lambda)t_n+1/2$, (\ref{eq:logist1}) is brought to the ``$c$ form''
$t_{n+1}=t_n^2+c$, $c=\lambda/2-\lambda^2/4$. The Mandelbrot set is defined
as (see e.g. \cite{Devaney})
\begin{equation}
  \label{eq:defM}
  \mathcal{M}=\{c: t_n \text{ bounded }\text{ if } t_0=0\}
\end{equation}
If $c\in\mathcal{M}$, then clearly $y_n$ in  (\ref{eq:infty1}) are bounded away from zero. Note that $t_0=0$ corresponds to $x_0=1/2$ implying $x_1=-\lambda/4$.
\item \label{(iii)}  $\mathcal{M}$ is a compact set; it coincides  
with the set of $c$ for which $J$ is connected. The
cardioid $\mathcal{H}=\{(2e^{it}-e^{2it})/4:t\in [0,2\pi)\}$ is contained in
$\mathcal{M}$; see \cite{Devaney}. This means $\{\lambda:|\lambda|<1\}$
corresponds to the interior of $\mathcal{M}$. We have
 $|\lambda|=1\Rightarrow c\in \partial \mathcal{M}\subset\mathcal{M}$.

\end{enumerate}

\subsection{Overview of Borel summability and transseries}\label{A4}

There is a vast literature on transseries, Borel summability, and
resurgence, see, for example \cite{Sauzin}. Most of the modern theory
originates in Ecalle's work \cite{Ecalle}.
\begin{Definition}\label{def1}
  We say that $f$ is given by a Borel summable transseries for
  $x>\nu$, if there exists a $\beta \in\CC$, a sequence $c_k$, with
  $\Re\,c_k\ge Ck$ for some $C>0$, and a sequence of functions $Y_k$,
  analytic in a neighborhood of $(0,\infty)$, having convergent Puiseux
  series at zero, and $|Y_k(p)|\le |B^ke^{\nu p}|$ (where $B$ and
  $\nu$ are independent of $k$) such that
\begin{equation}
  \label{eq:deftr}
  f(x)=\sum_{k=0}^{\infty}e^{-c_k x}x^{(k+1)\beta}\mathcal{L} Y_k
\end{equation}
where $\mathcal{L}$ is the usual Laplace transform:
\begin{equation}
  \label{eq:lapl}
  (\mathcal{L}Y)(x)=\int_0^{\infty}e^{-px}Y(p)dp
\end{equation}
The definition for other directions $\theta$ in the $x$ complex domain is obtained
by changing the variable to $x'=xe^{-i\theta}$.

The Borel-Laplace summation operator is denoted by $\mathcal{LB}$.
\end{Definition}
\begin{Definition}
  A formal power series in powers of $1/x$ is Borel summable as
  $x\to\infty$ if it is the asymptotic series of $\mathcal{L} Y$,
  where $Y$ is as in Definition \ref{def1}. (We note that by Watson's
  Lemma \cite{benderorszag}, $(\mathcal{L}Y)(x)$ has an asymptotic
  series as $x\to\infty$, which is the termwise Laplace transform of
  the Taylor series of $Y$ at zero.)
\end{Definition}
Transseries representations contain therefore manifest asymptotic information.
\begin{Definition}
  A function $Y(p)$ is resurgent in $p$ in the sense of Ecalle \cite{Ecalle}, if it
  is analytic on the Riemann surface of $\CC\setminus J$, where $J$ is
  a discrete set, and has uniform exponential bounds along any
  direction towards infinity cf. \cite{Sauzin}.\footnote{It is also required
that certain relations, called bridge equations, hold; in our
case it would be easy to derive them from the
explicit form of $H$, but we omit this calculation.} By abuse of language,
  $f(x)$ is called resurgent if it satisfies the requirements in
  Definition \ref{def1} and {\bf all} $Y_k$ are resurgent.
\end{Definition}
This is especially useful when global information about $f$ for
$x\in\CC$ is needed: deformation of contours in $p$, and collecting
residues when/if singularities are crossed, provides a straightforward
way to obtain this information.

\end{document}